\documentclass[11pt,a4paper,reqno,dvipsnames]{amsart}
\allowdisplaybreaks[4]

\usepackage{amsfonts}
\usepackage{amsmath}
\usepackage{amssymb}
\usepackage{amsthm}
\usepackage{amsxtra}
\usepackage{appendix}
\usepackage{bbm}
\usepackage{braket}
\usepackage{comment}
\usepackage{extarrows}
\usepackage{graphicx}
\usepackage{latexsym}
\usepackage{mathrsfs}
\usepackage{yhmath}
\usepackage{mathtools}

\usepackage{float}
\usepackage{booktabs}
\usepackage{fancyvrb}
\usepackage{multicol}
\usepackage{xcolor}
\usepackage[normalem]{ulem}

\usepackage{tikz-cd}
\usepackage{caption}
\usepackage[labelformat=simple]{subcaption}

\usepackage{hyperref} % hyperref must be loaded after natbib
\hypersetup{breaklinks, raiselinks}
\hypersetup{colorlinks, citecolor=blue, linkcolor=blue, urlcolor=blue}

\usepackage{bm}

\usepackage[initials,lite]{amsrefs} % amsrefs must be loaded after hyperref
\AtBeginDocument{%
    \def\MR#1{}
}
\usepackage{cleveref}
\Crefname{Lemma}{Lemma}{Lemmas}
\Crefname{Theorem}{Theorem}{Theorems}

\usepackage{fullpage}

\theoremstyle{plain}
\newtheorem{Theorem}{Theorem}[section]
\newtheorem{Lemma}[Theorem]{Lemma}
\newtheorem{Corollary}[Theorem]{Corollary}
\newtheorem{Proposition}[Theorem]{Proposition}

\theoremstyle{definition}

\newtheorem{Assumptions and Discussion}[Theorem]{Assumptions and Discussion}

\newtheorem{Definition}[Theorem]{Definition}

\newtheorem{Remark}[Theorem]{Remark}

\theoremstyle{remark}

\newtheorem{Setting}[Theorem]{Setting}

\newtheorem*{acknowledgment*}{Acknowledgment}

\usepackage[shortlabels]{enumitem}
\SetEnumerateShortLabel{a}{\textup{(\alph*)}}
\SetEnumerateShortLabel{A}{\textup{(\Alph*)}}
\SetEnumerateShortLabel{1}{\textup{(\arabic*)}}
\SetEnumerateShortLabel{i}{\textup{(\roman*)}}
\SetEnumerateShortLabel{I}{\textup{(\Roman*)}}
\def\bar#1{\overline{#1}}

\def\Char{\operatorname{char}}

\def\cochord{\operatorname{cochord}}

\def\conv{\operatorname{conv}}

\def\deg{\operatorname{deg}}
\def\depth{\operatorname{depth}}

\def\dim{\operatorname{dim}}

\def\floor#1{\left\lfloor #1 \right\rfloor}

\def\gr{\operatorname{gr}}

\def\ini{\operatorname{in}} % initial ideal/term

\def\into{\hookrightarrow}

\def\KK{{\mathbb K}}
\def\Mat{\operatorname{Mat}}

\def\Min{\operatorname{Min}}

\def\onto{\twoheadrightarrow}

\def\QQ{{\mathbb Q}}

\def\reg{\operatorname{reg}}

\def\ri{\operatorname{ri}}

\def\RR{{\mathbb R}}

\def\depth{\operatorname{depth}}

\def\ZZ{{\mathbb Z}}
\newcommand\bdalpha{\bm{\alpha}}
\newcommand\bdA{\bm{A}}

\newcommand\bdD{\bm{D}}

\newcommand\bdE{\bm{E}}
\newcommand\bde{\bm{e}}

\newcommand\bdF{\bm{F}}

\newcommand\bdu{\bm{u}}

\newcommand\bdv{\bm{v}}

\newcommand\bdX{\bm{X}}

\newcommand\bdY{\bm{Y}}

\newcommand\calA{\mathcal{A}}

\newcommand\calC{\mathcal{C}}

\newcommand\calF{\mathcal{F}}

\newcommand\calG{\mathcal{G}}

\newcommand\calJ{\mathcal{J}}

\newcommand\calQ{\mathcal{Q}}

\newcommand\calR{\mathcal{R}}

\newcommand\calZ{\mathcal{Z}}

\newcommand\frakm{\mathfrak{m}}

\newcommand\frakP{\mathfrak{P}}
\newcommand\frakp{\mathfrak{p}}

\newcommand{\Ass}{\operatorname{Ass}}

\newcommand{\match}{\operatorname{match}}
\newcommand{\pd}{\operatorname{pd}}

\begin{document}

\title{Powers of generalized binomial edge ideals of path graphs}
%\date{\today}

\author{Yi-Huang Shen}
\address{CAS Wu Wen-Tsun Key Laboratory of Mathematics, School of Mathematical Sciences, University of Science and Technology of China, Hefei, Anhui, 230026, P.R.~China}
\email{yhshen@ustc.edu.cn}

\author{Guangjun Zhu$^{\ast}$}
\address{School of Mathematical Sciences, Soochow University, Suzhou, Jiangsu, 215006, P.R.~China}
\email{zhuguangjun@suda.edu.cn}

\thanks{$^{\ast}$ Corresponding author}
% 13C15, Dimension theory, depth, related commutative rings (catenary, etc.)
% 13P10, Grobner bases; other bases for ideals and modules (e.g., Janet and border bases)
% 05E40, Combinatorial aspects of commutative algebra
% 13F20, Polynomial rings and ideals; rings of integer-valued polynomials


\begin{abstract}
    In this article, we study the powers of the generalized binomial edge ideal $\mathcal{J}_{K_m,P_n}$
    of a path graph $P_n$. We explicitly compute their regularities  and  determine the limit of their depths. We also show that these ordinary powers coincide with their symbolic powers. Additionally, we 
    study the Rees algebra and the special fiber ring of $\mathcal{J}_{K_m,P_n}$ via Sagbi basis theory. In particular, we obtain exact formulas for the regularity of these blowup algebras.
\end{abstract}

\maketitle

\section{Introduction}
Let $m$ and $n$ be two positive integers and $[m]$ be  the set $\{1,2,\ldots, m\}$. In addition, let $S=\KK[\bdX]\coloneqq\KK[x_{i,j}:i\in[m],j\in[n]]$ be the polynomial ring in $m\times n$ variables over a field $\KK$. In \cite{MR3290687}, Ene et al.~introduced the {binomial edge ideal of a pair of graphs}. Specifically, let $G_1$ and $G_2$ be simple graphs on vertex sets $[m]$ and $[n]$ respectively. Suppose that $e=\{i, j\}\in E(G_1)$ and $f=\{k, l\}\in E(G_2)$ are two edges with $i<j$ and $k<l$. Then, one can associate a $2$-minor 
\[
    p_{(e,f)}=[i,j\,|\,k,l]\coloneqq x_{i,k}x_{j,l}-x_{i,l}x_{j,k} 
\]
to the pair $(e, f)$. The \emph{binomial edge ideal of the pair $(G_1, G_2)$} is defined as follows
\[
    \calJ_{G_1,G_2}\coloneqq (p_{(e,f)} : e \in E(G_1), f \in E(G_2)).
\]
 This is a generalization of the  classical \emph{binomial edge ideals} in \cites{MR2669070, MR2782571}, if one of $G_1$ and $G_2$ is the complete graph $K_2$. Meanwhile, the \emph{ideals generated by adjacent minors} in \cite{MR1627343} turn out to be the binomial edge ideals of a pair of path graphs.

In the last decade, researchers have  tried to understand the connection between the algebraic properties of $\calJ_{G_1,G_2}$ and the combinatorial properties of $G_1$ and $G_2$. In \cite{MR3290687}, Ene et al.~proved that $\calJ_{G_1,G_2}$ is a radical ideal if and only if either $G_1$ or $G_2$ is a complete graph. Furthermore,  $\calJ_{G_1,G_2}$ is a prime ideal if and only if both $G_1$ and $G_2$ are complete. The importance of the complete graph in the pair can be further demonstrated by \cite[Theorem 3.1]{MR3859970} and \cite[Theorems 1 and 10]{MR3040610}.

If $K_m$ is a complete graph with $m$ vertices and $G$ is a simple graph, then $\calJ_{K_m,G}$ is the \emph{generalized binomial edge ideal} associated with $G$, which was previously introduced by Rauh in \cite{MR3011436} for the study of conditional independence ideals. Inspired by the progress on classical binomial edge ideals, researchers turn to the study of the generalized binomial edge ideals. For example, Chaudhry and Irfan  in \cite{MR4233116}  proved  that, for a block graph $G$, $\calJ_{K_m,G}$ is Cohen--Macaulay if and only if $\calJ_{K_m,G}$ is unmixed if and only if $G$ is a complete graph.

Blowup algebras are ubiquitous in commutative algebra and algebraic geometry. The \emph{Rees algebra} $\calR(I)\coloneqq\oplus_{k\ge 0}I^kT^k$ and the \emph{special fiber ring} $\calF(I)\coloneqq\calR(I)/\frakm\calR(I)$ of an ideal $I$ are two important blowup algebras  that  uniquely encode all powers of $I$.  In the case of binomial edge ideals, their properties have been explored by several researchers. For example, in \cite{MR4173994}, Jayanthan et al.~proved that the Rees algebra of an almost complete intersection binomial edge ideal is Cohen--Macaulay. In \cite{MR4425287},  Ene et al.~showed that  the Rees algebra of the binomial edge ideal of a  closed graph is Cohen--Macaulay. Later, Kumar in \cite{MR4405525}  proved that the special fiber ring of the binomial edge ideal of a closed graph is Koszul and  normal Cohen--Macaulay. 

In this article, we are interested in the Castelnuovo--Mumford regularity (regularity for short) and the depth of powers of 
generalized binomial edge ideals.
It is well-known that if $I$ is a homogeneous ideal of a polynomial ring $R$,
then $\reg(R/I^t)$ is asymptotically linear in $t$. At the same time, $\depth(R/I^t)$ is constant for sufficiently large $t$ (cf.~\cites{MR1711319,MR1621961,MR0530808}). It is usually difficult to determine when these phenomena begin. For this problem, the simplest case is when $I$ is a quadratic squarefree monomial ideal, i.e., when $I$ can be recognized as the edge ideal of a suitable graph.  In this case, many results have been achieved for simple classes such as  forest graphs, cycle graphs, bipartite graphs, and so on. In addition, a few results are known for the binomial edge ideal of a graph. For example, 
in \cite{JKS}, Jayanthan et al.~gave an upper bound on the regularity of powers of almost complete intersection binomial edge ideals using the quadratic sequence 
approach.  Meanwhile, they 
gave the exact formulas for the regularity of powers of binomial edge ideals of several simple graphs such as cycle graphs, star graphs, and balloon graphs. Recently, in \cite{MR4544259}, we gave explicit formulas for the regularity of powers of  binomial edge ideals which are almost complete intersections. 
At the same time, in \cite{MR4563443}, Wang and Tang studied the depth of powers of binomial edge ideals of complete bipartite graphs.  Ene et al.~in \cite{MR4425287} studied the regularity and the depth of powers of the binomial edge ideals of connected closed graphs. 

However, nearly nothing is known about the algebraic properties of powers of generalized binomial edge ideals. In this paper we will start such a study by considering the generalized binomial edge ideal $\calJ_{K_m,P_n}$ of the path graph $P_n$.

The article is organized as follows. In Section \ref{sec:prelim}, we briefly review essential definitions and terminology that we will need later. In Section \ref{sec:power}, we show that taking-initial-ideal commutes with taking-powers, when coming to $\calJ_{K_m,P_n}$. Using this fact, we study the regularity and the depth of the powers of $\calJ_{K_m,P_n}$ via the Sagbi basis theory. We also show that the symbolic powers and ordinary powers of $\calJ_{K_m,P_n}$ coincide. In Section \ref{sec:blowup}, we study the regularities of the Rees algebra and the special fiber ring of $\calJ_{K_m,P_n}$, by  considering the corresponding problems of their initial algebras. The final computation builds on the combinatorial optimization of different flavors. In the last section, we give applications of the previous results. Since the philosophy of combinatorial pure subrings 
applies here, we can consider the binomial edge ideal of a pair of graphs. In particular, we will give natural bounds on the regularities of the powers of $\calJ_{K_m,G}$ as well as the two blowup algebras of this generalized binomial edge ideal, if $G$ contains an induced path $P_n$.

\section{Preliminaries}
\label{sec:prelim}

In this brief section, we provide a concise overview of some combinatorial notions that will be employed throughout this paper. For a more comprehensive treatment from an algebraic perspective, we refer the readers to \cites{MR2724673,MR3838370,MR3362802}.

Let $G$ be a simple graph with the vertex set $V(G)$ and the edge set $E(G)$. For a vertex $v$ of $G$, the set of all neighborhoods of $v$ is denoted by $N_G(v)= \{u \in V(G): \{u,v\}\in E(G)\}$. A vertex $v$ is called a \emph{leaf} of $G$ if $N_G(v)$ has cardinality one and $v$ is \emph{isolated} if $N_G(v)=\emptyset$.  

For any subset $A$ of $V(G)$, let $G[A]$ denote the \emph{induced subgraph} of $G$ on the  set $A$, i.e., for $u,v \in A$, $\{u,v\} \in E(G[A])$ if and only if $\{u,v\}\in E(G)$. At the same time, we denote the induced subgraph of $G$ on the  set $V(G)\setminus A$ by $G\setminus A$. 

A subset $M\subset E(G)$ is a \emph{matching} of $G$ if $e\cap e'=\emptyset$ for all distinct edges $e$ and $e'$ in $M$.  The \emph{matching number} of $G$, denoted by $\match(G)$, is the maximum size of a  matching in $G$.  If $G$ is a bipartite graph having  vertex partitions $V_1$ and $V_2$, a \emph{complete matching from $V_2$ to $V_1$} is  a matching in which there is one edge incident with every vertex in $V_2$. In other words, every vertex in $V_2$ is matched against some vertex in $V_1$. Whence, $\match(G)=|V_2|$.

A \emph{walk} $W$ of length $n$ in a graph $G$ is a sequence of vertices $(w_1,\ldots, w_n,w_{n+1})$, 
such that $\{w_i,w_{i+1}\}\in E(G)$ for $1\le i\le n$. The walk $W$ is \emph{closed} if $w_1=w_{n+1}$.
Furthermore, the walk $W$ is called a \emph{cycle} if it is closed and the points $w_1,\ldots,w_n$ are distinct. At the same time, a \emph{path} is a walk where all  points are distinct. For simplicity, a path of length $n-1$ is denoted by $P_n$, and a cycle of length $n$ is denoted by $C_n$.

\section{Powers of generalized binomial edge ideals of paths}
\label{sec:power} 

In this section, we will study the regularity and the depth of the powers of the generalized binomial edge ideal $\calJ_{K_m,P_n}$, where $P_n$ is a path graph with $n$ vertices. Using Sagbi basis theory, we can turn to the corresponding study of their initial ideals. However, to make this approach work, we need to show first in \Cref{lem:sagbi_path} that taking-initial-ideal commutes with taking-powers when coming to $\calJ_{K_m,P_n}$. Since the proof for this is rather technical, we need  to make some preparations. Throughout this paper, we will stay with the following setting:

\begin{Setting}
    \label{set:path} 
    Let $m,n\ge 2$ be two integers. The polynomial ring $S \coloneqq \KK[x_{i,j}:i\in [m],j\in [n]]$ over a field $\KK$ is endowed with the term order $\tau$, which is the lexicographic order on $S$ induced by the natural order 
    \[
        x_{1,1}>x_{1,2}>\cdots>x_{1,n}>x_{2,1}>x_{2,2}>\cdots>x_{2,n}>\cdots>x_{m,1}>x_{m,2}>\cdots>x_{m,n}.
    \]
    Let $P_n$ be the path on the set $[n]$ whose edge set is $E(P_n)=\Set{\{i,i+1\}:i\in [n-1]}$. Furthermore, let $H$ be the graph on the set $\Set{x_{i,j}:i\in[m],j\in[n]}$ with 
    \[
        E(H)=\Set{\{x_{i,j},x_{i',j+1}\}:1\le i< i'\le m, 1\le j\le n-1}.
    \]
    For $k\in [n-1]$, let $H_k$ be the induced subgraph of $H$ on the set $\Set{x_{i,j}: i\in[m],j\in\{k,k+1\}}$. 
\end{Setting}

Regarding the graph $H$ in \Cref{set:path}, we have the following two observations:

\begin{Remark}
    \label{rmk:graph_H}
    \begin{enumerate}[a]
        \item If $G$ is a simple graph on the set $[n]$, the Gr\"obner basis of the generalized binomial edge ideal $\calJ_{K_m,G}$ with respect to the term order $\tau$ was computed in \cite[Theorem 2]{MR3011436}. In particular, the initial ideal $\ini_\tau(\calJ_{K_m,P_n})$ equals the ideal
            \[
                (x_{i,j}x_{i',j+1}:1\le i<i'\le m, 1\le j\le n-1) 
            \]
            in $S$. It is clear that this is the edge ideal $I(H)$ of the graph $H$.

        \item The graph $H$ is bipartite with respect to the bipartition $V_1\sqcup V_2$, where 
            \[
                V_1=\Set{x_{i,j}\in \bdX:i\in [m], \text{$j$ is odd}} 
                \qquad \text{and} \qquad
                V_2=\Set{x_{i,j}\in \bdX:i\in [m], \text{$j$ is even}}.
            \]
            When $m\ge 3$, $H$ has $3$ connected components, where $x_{m,1}$ and $x_{1,n}$ are isolated vertices. If instead $m=2$, then $H$ has $2+(n-1)=n+1$ connected components, where $x_{m,1}$ and $x_{1,n}$ are still isolated vertices; see \Cref{Fig:H}. 
            \begin{figure}[htbp]
                \begin{minipage}{0.48\textwidth}
                    \centering
                    \begin{tikzpicture}[thick, scale=0.8, every node/.style={scale=0.98}]]
                        \draw[solid](1,0)--(0,1);
                        \draw[solid](2,0)--(1,1);
                        \draw[solid](5,0)--(4,1);
                        \draw[solid](6,0)--(5,1);
                        \draw[dotted] (2.5,0.5)--(3.5,0.5);

                        \shade [shading=ball, ball color=black] (0,1) circle (.07);
                        \shade [shading=ball, ball color=black] (1,1) circle (.07);
                        \shade [shading=ball, ball color=black] (4,1) circle (.07);
                        \shade [shading=ball, ball color=black] (5,1) circle (.07);
                        \shade [shading=ball, ball color=black] (6,1) circle (.07);
                        \shade [shading=ball, ball color=black] (0,0) circle (.07);
                        \shade [shading=ball, ball color=black] (1,0) circle (.07);
                        \shade [shading=ball, ball color=black] (2,0) circle (.07);
                        \shade [shading=ball, ball color=black] (5,0) circle (.07);
                        \shade [shading=ball, ball color=black] (6,0) circle (.07);
                    \end{tikzpicture}
                    \subcaption*{$m=2$ case}
                \end{minipage}\hfill
                \begin{minipage}{0.48\textwidth}
                    \centering
                    \begin{tikzpicture}[thick, scale=0.8, every node/.style={scale=0.98}]]
                        \draw[solid](1,0)--(0,1);
                        \draw[solid](2,0)--(1,1);
                        \draw[solid](5,0)--(4,1);
                        \draw[solid](6,0)--(5,1);

                        \draw[solid](1,0)--(0,2);
                        \draw[solid](2,0)--(1,2);
                        \draw[solid](5,0)--(4,2);
                        \draw[solid](6,0)--(5,2);

                        \draw[solid](1,1)--(0,2);
                        \draw[solid](2,1)--(1,2);
                        \draw[solid](5,1)--(4,2);
                        \draw[solid](6,1)--(5,2);

                        \draw[solid](1,0)--(0,3);
                        \draw[solid](2,0)--(1,3);
                        \draw[solid](5,0)--(4,3);
                        \draw[solid](6,0)--(5,3);

                        \draw[solid](1,2)--(0,3);
                        \draw[solid](2,2)--(1,3);
                        \draw[solid](5,2)--(4,3);
                        \draw[solid](6,2)--(5,3);

                        \draw[solid](1,1)--(0,3);
                        \draw[solid](2,1)--(1,3);
                        \draw[solid](5,1)--(4,3);
                        \draw[solid](6,1)--(5,3);

                        \draw[dotted] (2.5,0.5)--(3.5,0.5);
                        \draw[dotted] (2.5,1.5)--(3.5,1.5);
                        \draw[dotted] (2.5,2.5)--(3.5,2.5);

                        \shade [shading=ball, ball color=black] (0,3) circle (.07);
                        \shade [shading=ball, ball color=black] (1,3) circle (.07);
                        \shade [shading=ball, ball color=black] (4,3) circle (.07);
                        \shade [shading=ball, ball color=black] (5,3) circle (.07);
                        \shade [shading=ball, ball color=black] (6,3) circle (.07);
                        \shade [shading=ball, ball color=black] (0,2) circle (.07);
                        \shade [shading=ball, ball color=black] (1,2) circle (.07);
                        \shade [shading=ball, ball color=black] (4,2) circle (.07);
                        \shade [shading=ball, ball color=black] (5,2) circle (.07);
                        \shade [shading=ball, ball color=black] (6,2) circle (.07);
                        \shade [shading=ball, ball color=black] (0,1) circle (.07);
                        \shade [shading=ball, ball color=black] (1,1) circle (.07);
                        \shade [shading=ball, ball color=black] (2,1) circle (.07);
                        \shade [shading=ball, ball color=black] (4,1) circle (.07);
                        \shade [shading=ball, ball color=black] (5,1) circle (.07);
                        \shade [shading=ball, ball color=black] (6,1) circle (.07);
                        \shade [shading=ball, ball color=black] (0,0) circle (.07);
                        \shade [shading=ball, ball color=black] (1,0) circle (.07);
                        \shade [shading=ball, ball color=black] (2,0) circle (.07);
                        \shade [shading=ball, ball color=black] (5,0) circle (.07);
                        \shade [shading=ball, ball color=black] (6,0) circle (.07);
                    \end{tikzpicture}
                    \subcaption*{$m=4$ case}
                \end{minipage}
                \caption{Graph $H$}
                \label{Fig:H}
            \end{figure}
    \end{enumerate}
\end{Remark}

The proof of \Cref{lem:sagbi_path} depends on the presentation ideal of the Rees algebra $\calR(I(H))$ of the edge ideal $I(H)$ and the ``lifts'' of the Sagbi basis. Suppose that $\calG(I(H))=\{f_1,\ldots,f_q\}$ is the minimal monomial generating set of $I(H)$. Then, there exists a canonical homomorphism from the polynomial ring  $B\coloneqq S[T_1,\dots,T_q]$ to the Rees algebra $\calR(I(H))\coloneqq S[f_1T,\dots,f_qT]\subset S[T]$, induced by $T_i\mapsto f_iT$.
Let $\deg(T_1)=\cdots=\det(T_q)=\deg(T)=1$, and $\deg(x_{i,j})=0$ for every $x_{i,j}\in \bdX$. Then, this map is a graded homomorphism of $S$-algebras.
Its kernel $J$ will be called the \emph{presentation ideal} of $\calR(I(H))$ with respect to $\calG(I(H))$. This is a graded ideal and $\calR(I(H))$ has the \emph{presentation} $\calR(I(H))=B/J$. 

Similarly, there is a canonical homomorphism from the polynomial ring $B'\coloneqq \KK[T_1,\dots,T_q]$ to the special fiber ring $\calF(I(H))\cong \KK[H]$, induced by $T_i\mapsto f_i$. The kernel ideal $J'$ of this map is called the \emph{presentation ideal} of $\calF(I)$. It leads to the \emph{presentation} $\calF(I(H))=B'/J'$.

In addition, let $W=(w_1,w_2,\dots,w_{2s+1}=w_1)$ be an even closed walk in $H$ and suppose that $e_j=\{w_{j},w_{j+1}\}$ for each $j$. We will write $T_{W^+}-T_{W^-}$ for the binomial $T_{e_1}T_{e_3}\cdots T_{e_{2s-1}}-T_{e_2}T_{e_4}\cdots T_{e_{2s}}$ in $J$. We are mostly interested in the case where $W$ is a primitive cycle. Recall that a \emph{cycle} $C$ is a closed walk with distinct vertices. A \emph{chord} of a cycle $C$ in a graph $G$ is an edge of $G$ that connects two non-adjacent vertices of $C$. A cycle without chords is called \emph{primitive}. Binomials from primitive cycles are important for describing the presentation ideals.

\begin{Lemma}
    [{\cite[Proposition 10.1.14, Theorem 10.1.15]{MR3362802}}]
    \label{lem:defining_ideal}
    Let $I$ be the edge ideal of a bipartite graph. 
    \begin{enumerate}[a]
        \item Suppose that $\calR(I) = B/J$ is the presentation of the Rees algebra $\calR(I)$.
            Then $J = BJ_1 + B P$, where $J_1$ is the degree $1$ part of the graded ideal $J$, and
            \[
                P\coloneqq \Set{T_{w^+}-T_{w^-}:\text{$w$ is a primitive cycle in $H$}}.
            \]
        \item Suppose that $\calF(I) = B'/J'$ is the presentation of the special fiber ring $\calF(I)$. Then, $J'$ is minimally generated by the set $P$.
    \end{enumerate}
\end{Lemma}

Recall that a bipartite graph is \emph{chordal bipartite} if every cycle of length at least six has a chord in it. In other words, the length of every primitive cycle of this bipartite graph is $4$.

\begin{Lemma}
    \label{lem:chordal_bipartite}
    Let $H,H_1,\dots,H_{n-1}$ be as in \Cref{set:path}. Then, $H$ is a chordal bipartite graph. 
\end{Lemma}
\begin{proof}
    Let $C=(c_1,c_2,\dots,c_{2s+1}=c_1)$ be a primitive cycle in $H$ with $s\ge 2$. By abuse of notation, if $k>2s$, we will identify $c_k$ with $c_{k-2s}$, and if $k<1$, we will identify $c_k$ with $c_{k+2s}$. We have the following two cases. \Cref{Fig:3_4} is helpful in understanding the arguments.
    \begin{enumerate}[a]
        \item \label{pf:chordal_bipartite_a}
            Suppose that $C$ is a cycle in $H_k$ for some $k$. Without loss of generality, we can assume that $k=1$ and $c_p=x_{i_p,j_p}$ for each $p$. Suppose also that $j_p=1$ if $p$ is odd, and $j_p=2$ if $p$ is even. We can also assume that $i_1=\min\{i_1,i_3,\dots,i_{2s-1}\}$. Then, $i_1<i_2$ and $i_1<i_{2s}$. By symmetry, we assume that $i_{2s}<i_{2}$. Note that $i_{2s-1}<i_{2s}$ at this time. Therefore, if $s>2$, then $C$ has a chord $\{x_{i_{2s-1},1},x_{i_2,2}\}$, which violates the primitivity of $C$.

        \item \label{pf:chordal_bipartite_b} 
            Suppose there is no $k$ such that $C$ is a cycle in $H_k$. Then, without loss of generality, we can assume that $c_1=x_{i_1,1}$. For each $k$ with $c_k=x_{i_k,1}$, we say that $k$ is \emph{marginal}. Note that in this case,
            we have $c_{k\pm 1}=x_{i_{k\pm 1},2}$ with $i_{k-1}\ne i_{k+1}$.  
            If $i_{k-1}<i_{k+1}$, we  say that $k$ is \emph{extendable} if $c_{k+2}=x_{i_{k+2},3}$. If instead $i_{k-1}>i_{k+1}$, we  say that $k$ is \emph{extendable} if $c_{k-2}=x_{i_{k-2},3}$. 
            \begin{enumerate}[i]
                \item Suppose  there is a marginal $k$ that is extendable. Without loss of generality, we assume that $i_{k-1}<i_{k+1}$. Note that  in this case $i_k<i_{k-1}<i_{k+1}<i_{k+2}$. If $s>2$, then $C$ has a chord $\{x_{i_{k-1},2},x_{i_{k+2},3}\}$,  which violates the primitivity of $C$.
                \item Suppose that there is no marginal $k$ that is extendable. Then $C$ is a cycle in $H_1$, a contradiction.  
            \end{enumerate}
    \end{enumerate}
    In short, $s=2$ and $C$ is a cycle of length $4$. \qedhere
\end{proof}

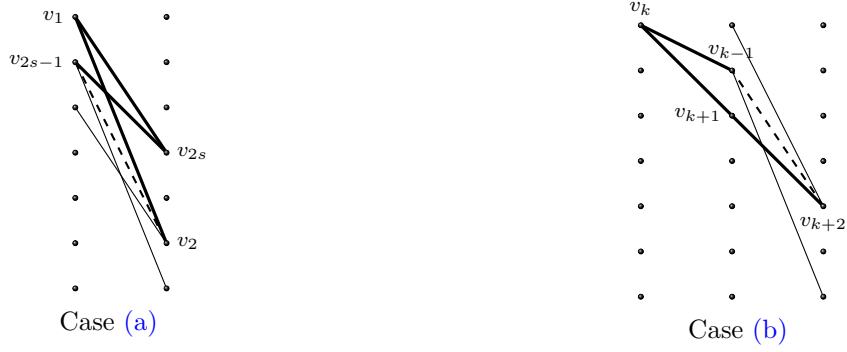
\begin{figure}[htbp]
    \begin{minipage}{0.48\textwidth}
        \centering
        \begin{tikzpicture}[thick, scale=0.6, every node/.style={scale=0.98}]]
            \draw[very thick](0,6)--(2,1);
            \draw[thin](0,4)--(2,1);
            \draw[dashed](0,5)--(2,1);
            \draw[very thick](0,6)--(2,3);
            \draw[very thick](0,5)--(2,3);
            \draw[thin](0,5)--(2,0);

            \shade [shading=ball, ball color=black] (0,0) circle (.07);
            \shade [shading=ball, ball color=black] (0,1) circle (.07);
            \shade [shading=ball, ball color=black] (0,2) circle (.07);
            \shade [shading=ball, ball color=black] (0,3) circle (.07);
            \shade [shading=ball, ball color=black] (0,4) circle (.07);
            \shade [shading=ball, ball color=black] (0,5) circle (.07) node [left] {\scriptsize$v_{2s-1}$};
            \shade [shading=ball, ball color=black] (0,6) circle (.07) node [left] {\scriptsize$v_{1}$};
            \shade [shading=ball, ball color=black] (2,0) circle (.07);
            \shade [shading=ball, ball color=black] (2,1) circle (.07) node [right] {\scriptsize$v_{2}$};
            \shade [shading=ball, ball color=black] (2,2) circle (.07);
            \shade [shading=ball, ball color=black] (2,3) circle (.07) node [right] {\scriptsize$v_{2s}$};
            \shade [shading=ball, ball color=black] (2,4) circle (.07);
            \shade [shading=ball, ball color=black] (2,5) circle (.07);
            \shade [shading=ball, ball color=black] (2,6) circle (.07);
        \end{tikzpicture}
        \subcaption*{Case \ref{pf:chordal_bipartite_a}}
    \end{minipage}\hfill
    \begin{minipage}{0.48\textwidth}
        \centering
        \begin{tikzpicture}[thick, scale=0.6, every node/.style={scale=0.98}]]
            \draw[very thick](0,6)--(2,5);
            \draw[very thick](0,6)--(2,4);
            \draw[very thick](4,2)--(2,4);
            \draw[thin](4,2)--(2,6);
            \draw[thin](4,0)--(2,5);
            \draw[dashed](4,2)--(2,5);

            \shade [shading=ball, ball color=black] (0,0) circle (.07);
            \shade [shading=ball, ball color=black] (0,1) circle (.07);
            \shade [shading=ball, ball color=black] (0,2) circle (.07);
            \shade [shading=ball, ball color=black] (0,3) circle (.07);
            \shade [shading=ball, ball color=black] (0,4) circle (.07);
            \shade [shading=ball, ball color=black] (0,5) circle (.07);
            \shade [shading=ball, ball color=black] (0,6) circle (.07) node [above] {\scriptsize$v_{k}$};
            \shade [shading=ball, ball color=black] (2,0) circle (.07);
            \shade [shading=ball, ball color=black] (2,1) circle (.07);
            \shade [shading=ball, ball color=black] (2,2) circle (.07);
            \shade [shading=ball, ball color=black] (2,3) circle (.07);
            \shade [shading=ball, ball color=black] (2,4) circle (.07) node [left] {\scriptsize$v_{k+1}$};
            \shade [shading=ball, ball color=black] (2,5) circle (.07) node [above] {\scriptsize$v_{k-1}$};
            \shade [shading=ball, ball color=black] (2,6) circle (.07);
            \shade [shading=ball, ball color=black] (4,0) circle (.07);
            \shade [shading=ball, ball color=black] (4,1) circle (.07);
            \shade [shading=ball, ball color=black] (4,2) circle (.07) node [below] {\scriptsize$v_{k+2}$};
            \shade [shading=ball, ball color=black] (4,3) circle (.07);
            \shade [shading=ball, ball color=black] (4,4) circle (.07);
            \shade [shading=ball, ball color=black] (4,5) circle (.07);
            \shade [shading=ball, ball color=black] (4,6) circle (.07);
        \end{tikzpicture}
        \subcaption*{Case \ref{pf:chordal_bipartite_b}}
    \end{minipage}
    \caption{Basic patterns}
    \label{Fig:3_4}
\end{figure}

The following result is indispensable for establishing the regularity result in \Cref{thm:power_reg_path}. It also provides a class of nice ideals sought in \cite[Section 1]{MR1477608}. 

\begin{Theorem}
    \label{lem:sagbi_path}
    Under the assumptions in \Cref{set:path},
    we have $(\ini_{\tau}(\calJ_{K_m,P_n}))^t=\ini_{\tau}(\calJ_{K_m,P_n}^t)$ for all $t\ge 1$.
\end{Theorem}
\begin{proof}
    We define a term order $\tau'$ on $S[T]$ as follows: for any monomials $u$ and $v$ in $S$, and non-negative integers $i$ and $j$, we set
    \[
        uT^i <_{\tau'} vT^j \qquad \Leftrightarrow \qquad i<j\quad \text{or}\quad i=j \quad \text{and} \quad u<_{\tau} v.
    \]
    It follows from \cite[Theorem 2.7]{MR1390693} that $(\ini_{\tau}(\calJ_{K_m,P_n}))^t=\ini_{\tau}(\calJ_{K_m,P_n}^t)$  for all $t\ge 1$, if and only if $\calR(\ini_{\tau}(\calJ_{K_m,P_n}))=\ini_{\tau'}(\calR(\calJ_{K_m,P_n}))$. In other words, the set 
    \[
        \calG\coloneqq \bdX \cup \left\{[i,j\,|\,k,k+1]T:1\le i<j\le m\text{ and }k\in [n-1]\right\} 
    \]
    forms a \emph{Sagbi basis} of $\calR(\calJ_{K_m,P_n})$ with respect to $\tau'$. 
    For brevity, write $\calG=\{g_1,\dots,g_p\}$, and let
    $F_1,\dots,F_q$ be a system of binomial relations of the affine semigroup ring $\KK[\ini_{\tau'}(g):g\in \calG]$. It follows from \cite[Proposition 1.1]{MR1390693} that we have to find $\lambda_\alpha^{(j)}\in \KK$ such that 
    \[
        F_j(g_1,\dots,g_p)=\sum_{\bdalpha}\lambda_\alpha^{(j)}g^{\bdalpha} \qquad \text{with $\ini_{\tau'}(g^{\bdalpha})\le \ini_{\tau'}(F_j(g_1,\dots,g_p))$}
    \]
    for all $\lambda_{\bdalpha}^{(j)}\ne 0$, where, as usual, $g^{\bdalpha}\coloneqq g_1^{\alpha_1}\cdots g_p^{\alpha_p}$ for a multi-index $\bdalpha=(\alpha_1,\dots,\alpha_p)$. Note that to verify this condition, it suffices to show that all the $\ini_{\tau'}(g^{\bdalpha})$ are distinct. We can check with ease  that this is satisfied in the following \cref{eqn:check_sagbi_1,eqn:check_sagbi_2,eqn:check_sagbi_3,eqn:check_sagbi_4,eqn:check_sagbi_5}.

    Note that $\KK[\ini_{\tau'}(g):g\in \calG]\cong \calR(I(H))$ for the graph $H$ introduced in \Cref{set:path}. It follows from \Cref{lem:defining_ideal} that we have the following two cases:
    \begin{enumerate}[A]
        \item First, suppose that $F=F_j$ is a binomial generator in $J_1$. Since the edge ideal $I(H)$ is
            quadratic, we have two subcases. 
            \begin{enumerate}[i]
                \item Suppose $F=uvT_{e_2}-u'v'T_{e_1}$ such that $e_1\coloneqq\{u,v\}$, $e_2\coloneqq\{u',v'\}$ and $e_1\cap e_2=\emptyset$. We can assume that $u=x_{i,k}$, $v=x_{j,k+1}$, $u'=x_{i',k'}$  and $v'=x_{j',k'+1}$ with $1\le i<j\le m$, $1\le i'<j'\le m$, and $k,k'\in [n-1]$. In this subcase, we have the simple equality:
                    \begin{align}
                        &x_{i,k}x_{j,k+1}[i',j'\,|\,k',k'+1]-x_{i',k'}x_{j',k'+1}[i,j\,|\,k,k+1] \notag \\
                        =\, &x_{i,k+1}x_{j,k}[i',j'\,|\,k',k'+1]-x_{i',k'+1}x_{j',k'}[i,j\,|\,k,k+1]. \label{eqn:check_sagbi_1}
                    \end{align}
                \item Suppose  $uT_{e_2}-vT_{e_1}$ such that $e_1\coloneqq \{u,r\}$ and $e_2\coloneqq \{v,r\}$. We have three additional subcases.
                    \begin{enumerate}[a]
                        \item Suppose  $r=x_{i_1,k}$, $u=x_{i_2,k+1}$, and $v=x_{i_3,k+1}$, with $1\le i_1<i_2<i_3\le m$ and $k\in [n-1]$. In this subcase, we have the simple equality:
                            \begin{equation}
                                \qquad
                                \qquad
                                \qquad
                                x_{i_2,k} [i_1,i_3\,|\,k,k+1]- x_{i_1,k} [i_2,i_3\,|\,k,k+1]= x_{i_3,k} [i_1,i_2\,|\,k,k+1].
                                \label{eqn:check_sagbi_2}
                            \end{equation}
                        \item Suppose  $r=x_{i_1,k+1}$, $u=x_{i_2,k}$, and $v=x_{i_3,k}$, where $1\le i_1<i_2<i_3\le m$ and $k\in [n-1]$. This subcase is similar to the previous one. 
                        \item Suppose  $u=x_{i_1,k}$, $r=x_{i_2,k+1}$, and $v=x_{i_3,k+2}$, with $1\le i_1<i_2<i_3\le m$ and $k\in [n-2]$. In this subcase, we have the simple equality:
                            \begin{align}
                                &x_{i_3,k+2}[i_1,i_2\,|\,k,k+1]-x_{i_1,k}[i_2,i_3\,|\,k+1,k+2] \notag \\
                                =\, &x_{i_3,k}[i_1,i_2\,|\,k+1,k+2]+x_{i_2,k+2}[i_1,i_3\,|\,k,k+1]\notag \\
                                &- x_{i_2,k}[i_1,i_3\,|\,k+1,k+2]-x_{i_1,k+2}[i_2,i_3\,|\,k,k+1].
                                \label{eqn:check_sagbi_3}
                            \end{align}
                    \end{enumerate} 
            \end{enumerate} 
        \item \label{lem:sagbi_path_B}
            Second, we assume that $F=T_{C^{+}}-T_{C^{-}}$ for some primitive cycle $C$. It follows from \Cref{lem:chordal_bipartite} that we can assume that $C=(c_1,c_2,c_3,c_4,c_5=c_1)$ in $H$.  Without loss of generality, we may assume that $c_1=x_{i_1,1}$.
            Then, we have two subcases.

            \begin{enumerate}[i]
                \item Suppose that $C$ is not a cycle in $H_1$. In this subcase, we can assume that $c_2=x_{i_2,2}$, $c_{3}=x_{i_3,3}$, and $c_4=x_{i_4,2}$ such that $1\le i_1<i_2<i_4<i_3\le m$.  Now, we have the simple equality:
                    \begin{align}
                        &[i_1,i_2\,|\,1,2][i_4,i_3\,|\,2,3]-[i_1,i_4\,|\,1,2][i_2,i_3\,|\,2,3] \notag \\ 
                        = \,  &-[i_1,i_2\,|\,2,3][i_4,i_3\,|\,1,2]+[i_1,i_4\,|\,2,3][i_2,i_3\,|\,1,2] \notag \\
                        &-[i_2,i_4\,|\,1,2][i_1,i_3\,|\,2,3] -[i_2,i_4\,|\,2,3][i_1,i_3\,|\,1,2]. 
                        \label{eqn:check_sagbi_4}
                    \end{align}

                \item Suppose that $C$ is a cycle in $H_1$. In this subcase, we can assume that $c_2=x_{i_2,2}$, $c_{3}=x_{i_3,1}$, and $c_4=x_{i_4,2}$ such that $1\le i_1<i_3<i_2<i_4\le m$. Now, we have the well-known Pl\"ucker relation:
                    \begin{equation}
                        \qquad 
                        [i_1,i_2\,|\,1,2][i_3,i_4\,|\,1,2]-[i_1,i_4\,|\,1,2][i_3,i_2\,|\,1,2] = [i_1,i_3\,|\,1,2][i_2,i_4\,|\,1,2].
                        \label{eqn:check_sagbi_5}
                    \end{equation}
                    This completes the proof. \qedhere
            \end{enumerate}
    \end{enumerate} 
\end{proof}

In the previous proof, we introduced a term order $\tau'$ on $S[T]$, which is induced from the term order $\tau$ on $S$. We will adopt it from now on. In particular, we can talk about the initial algebra $\ini_{\tau'}(\calR(\calJ_{K_m,P_n}))$.

\begin{Corollary}
    \label{cor:CM} 
    \begin{enumerate}[a]
        \item The Rees algebras $\calR(\calJ_{K_m,P_n})$ and $\calR(\ini_\tau(\calJ_{K_m,P_n}))$ are normal Cohen--Macaulay domains. In particular, if $\Char(\KK)=0$, then $\calR(\calJ_{K_m,P_n})$ has rational singularities. If instead $\Char(\KK)>0$, then $\calR(\calJ_{K_m,P_n})$ is F-rational.

        \item \label{limitdepth_a} 
            Let $\calG=\{g_1,\ldots,g_q\}$ be the natural generating set of $\calJ_{K_m,P_n}$. Then $\calG$ is a Sagbi basis of the $\KK$-subalgebra $\KK[g_1,\ldots,g_q]$ of $S$ with respect to the lexicographic order $\tau$ on $S$, i.e.,
            \[
                \ini_\tau(\KK[g_1,\ldots,g_q])=\KK[\ini_\tau(g_1),\ldots,\ini_\tau(g_q) ].
            \]
        \item 
            The special fiber rings $\calF(\calJ_{K_m,P_n})$ and $\calF(\ini_\tau(\calJ_{K_m,P_n}))$ are normal Cohen--Macaulay domains. In particular, $\calF(\calJ_{K_m,P_n})$ has rational singularities if $\KK$ is of characteristic $0$, and $\calF(\calJ_{K_m,P_n})$ is F-rational if $\KK$ is of positive characteristic. 
        \item \label{cor:CM_analytic_spread}
            The analytic spreads of $\calJ_{K_m,P_n}$ and $\ini_\tau(\calJ_{K_m,P_n})$ coincide. They are given by
            \[
                \begin{cases}
                    n-1, &\text{if $m=2$,}\\
                    mn-3, & \text{if $m\ge 3$.}
                \end{cases}
            \]
        \item\label{cor:CM-d} The associated graded rings $\gr_{\calJ_{K_m,P_n}}(S)$ and  $\gr_{\ini_\tau(\calJ_{K_m,P_n})}(S)$ are  Cohen--Macaulay. 
    \end{enumerate} 
\end{Corollary}
\begin{proof}
    \begin{enumerate}[a]
        \item Since $\ini_\tau(\calJ_{K_m,P_n})=I(H)$ is the edge ideal of the bipartite graph $H$ constructed in \Cref{set:path}, the Rees algebra $\calR(\ini_\tau(\calJ_{K_m,P_n}))$ is normal by \cite[Proposition 10.5.8]{MR3362802}. It is then Cohen--Macaulay by \cite{MR304376}. Since $\ini_{\tau'}(\calR(\calJ_{K_m,P_n}))=\calR(\ini_\tau(\calJ_{K_m,P_n}))$ by \Cref{lem:sagbi_path} and \cite[Theorem 2.7]{MR1390693}, it follows from \cite[Corollary 2.3]{MR1390693}  that $\calR(\calJ_{K_m,P_n})$ has rational singularities if $\Char(\KK)=0$, and  $\calR(\calJ_{K_m,P_n})$ is F-rational if $\Char(\KK)>0$. Furthermore, the normality of $\ini_\tau(\calR(\calJ_{K_m,P_n}))$ implies that $\calR(\calJ_{K_m,P_n})$ is a normal Cohen--Macaulay domain again by \cite[Corollary 2.3]{MR1390693}. 
        \item This part follows from \Cref{lem:defining_ideal} and the part \ref{lem:sagbi_path_B} of the proof of  \Cref{lem:sagbi_path}. 
        \item It follows from the previous part \ref{limitdepth_a} that the special fiber ring $\calF(\ini_\tau(\calJ_{K_m,P_n}))=\ini_\tau(\calF(\calJ_{K_m,P_n}))$ is isomorphic to the edge subring  $\KK[H]$ of the graph $H$ constructed in Setting \ref{set:path}. Since $H$ is bipartite, $\KK[H]$ is  Koszul, normal and Cohen--Macaulay by \cite[Theorem 1]{MR1657721} and \cite[Proposition 10.3.1]{MR3362802}. The remaining parts follow again from \cite[Corollary 2.3]{MR1390693}.
        \item Recall that the \emph{analytic spread} of a graded ideal $I$ is the (Krull) dimension of the special fiber ring $\calF(I)$. By \cite[Proposition 2.4]{MR1390693}, the Hilbert functions of $\calF(\calJ_{K_m,P_n})$ and $\ini_\tau(\calF(\calJ_{K_m,P_n}))$ coincide. Therefore, they have the same dimension. Meanwhile, $\ini_\tau(\calF(\calJ_{K_m,P_n}))=\calF(\ini_\tau(\calJ_{K_m,P_n}))$ by \ref{limitdepth_a}. Thus, the analytic spreads of $\calJ_{K_m,P_n}$ and $\ini_\tau(\calJ_{K_m,P_n})$ coincide. Note that $\calF(\ini_\tau(\calJ_{K_m,P_n}))=\KK[H]$ and $H$ is a bipartite graph. Therefore, it remains to apply \Cref{rmk:graph_H} and \cite[Corollary 10.1.21]{MR3362802}.

        \item Since the Rees algebras $\calR(\calJ_{K_m,P_n})$ and $\calR(\ini_\tau(\calJ_{K_m,P_n}))$ are Cohen--Macaulay, this part follows from \cite[Proposition 1.1]{MR638557}. \qedhere
    \end{enumerate} 
\end{proof}

Let $I$ be an ideal in $S$. Recall that $I$ satisfies the \emph{persistence property} if and only if the sets of associated primes satisfy $\Ass(I^t) \subseteq \Ass(I^{t+1})$ for all $t\ge 1$. In addition, the ideal $I$ has the \emph{strong persistence property} if and only if $I^{t+1} : I = I^t$ for all $t\ge 1$. It is not difficult to see that the strong persistence property implies the persistence property. On the other hand, the converse is also known to be false.

\begin{Corollary}
    Under the assumptions in \Cref{set:path}, the quadratic ideals
    $\ini_\tau(\calJ_{K_m,P_n})$ and $\calJ_{K_m,P_n}$ satisfy the strong persistence property. In particular, both $\ini_\tau(\calJ_{K_m,P_n})$ and $\calJ_{K_m,P_n}$ satisfy the persistence property.
\end{Corollary}
\begin{proof}
    Since $\ini_\tau(\calJ_{K_m,P_n})=I(H)$ is the edge ideal of the  bipartite graph $H$ constructed in \Cref{set:path},  $\ini_\tau(\calJ_{K_m,P_n})$ satisfies the strong persistence property by \cite[Lemma 7.7.11]{MR3362802}. Furthermore, by \Cref{lem:sagbi_path}, we have $\ini_{\tau}(\calJ_{K_m,P_n}^t)= (\ini_{\tau}(\calJ_{K_m,P_n}))^t$ for all $t\ge 1$. Thus, $\calJ_{K_m,P_n}$ also satisfies the strong persistence property by \cite[Proposition 3.13]{MR4425287}.
\end{proof}

\begin{Remark}
    The Sagbi-basis property mentioned in \Cref{lem:sagbi_path} is not a common characteristic for generalized binomial edge ideals. Even for the complete graph $K_3$, one has $\ini_\tau(\calJ_{K_3,K_3}^2)\ne (\ini_\tau(\calJ_{K_3,K_3}))^2$. In other words, the natural generators of the generalized binomial edge ideal $\calJ_{K_3,K_3}$ do not form a Sagbi basis under the given term order. 
\end{Remark}

As further applications of \Cref{lem:sagbi_path}, we study the regularity and the depth of powers of $\calJ_{K_m,P_n}$ in \Cref{thm:power_reg_path} and \Cref{limitdepth} respectively. 
In particular, we give a closed formula of the regularities.
To achieve it, we apply the following lemma: 

\begin{Lemma}
    [{\cite[Theorem 4.4(2)]{MR4210998}}]
    \label{lem:reg_cochord}
    Let $G$ be a graph. Then, for all $t\ge 1$  we have
    \[
        \reg(I(G)^t)\le 2t+\cochord(G)-1.
    \]
\end{Lemma}

In particular, we need to compute the co-chordal cover number of the  bipartite graph $H$ from \Cref{set:path}. For this purpose, recall that a graph $G$ is \emph{chordal} if every induced cycle in $G$ has length $3$. Relatedly, a \emph{perfect elimination order} of a graph $G$ is an order $v_1,\dots,v_n$ of its vertices such that for all $i\in [n]$, the neighbor $N_{G_i}(v_i)$ of $v_i$, restricted to the induced subgraph $G_i$ of $G$ on the set $\{v_{i},\ldots,v_n\}$, induces a complete subgraph in $G$. It is well-known that a graph is chordal if and only if it admits a perfect elimination order.

On the other hand, a graph $G$ is \emph{co-chordal} if its complement graph $G^\complement$ is chordal. The \emph{co-chordal cover number} of a graph $G$, denoted $\cochord(G)$, is the minimum number $k$ such that there exist co-chordal subgraphs $G_1, \ldots, G_k$ of $G$ with $E(G) = \bigcup_{i=1}^k E(G_i)$.

In the following, we compute the co-chordal cover number of the  bipartite graph $H$. An upper bound on this number is sufficient for our application.

\begin{Lemma}
    \label{lem:cochord} 
    For the graph $H$ introduced in \Cref{set:path}, one has $\cochord(H)\le n-1$.
\end{Lemma}
\begin{proof}
    It is clear that the induced subgraphs $H_1,H_2,\dots,H_{n-1}$ are  pairwise isomorphic. Furthermore, $E(H)=\bigcup_{k=1}^{n-1}E(H_k)$. It remains to show that $H_1$ is co-chordal. Notice that
    \begin{align*}
        E(H_1^\complement)&=\Set{\{x_{i,1},x_{j,1}\}:1\le i<j\le m} \cup \Set{\{x_{i,2},x_{j,2}\}:1\le i<j\le m} \\
        &\quad \cup \Set{\{x_{i,1},x_{j,2}\}:1\le j\le i \le m}.
    \end{align*}
    It can be directly verified  that $x_{1,1},x_{2,1},\dots,x_{m,1},x_{1,2},x_{2,2},\dots,x_{m,2}$ form a perfect elimination order. Therefore, $H_1^{\complement}$ is a chordal graph.
\end{proof}

The proof of the following \Cref{thm:power_reg_path} will also use the comparison in \Cref{power}. The argument for the latter is logically irrelevant to what is presented here. Therefore, we decide to postpone the discussion of \Cref{power} to a later section of our paper, in order to maintain the consistency and coherence of our presentation.

\begin{Theorem}
    \label{thm:power_reg_path}
    Under the assumptions in \Cref{set:path},
    we have
    \[
        \reg\left(\frac{S}{\calJ^t_{K_m,P_n}}\right)=\reg\left(\frac{S}{\ini_\tau(\calJ^t_{K_m,P_n})}\right)=2(t-1)+(n-1)
    \]
    for all $t\ge 1$.
\end{Theorem}
\begin{proof}
    It follows from  
    \cite[Theorem 3.1]{MR4425287}, \cite[Theorem 3.3.4]{MR2724673}, and \Cref{power} 
    that
    \[
        \reg\left(\frac{S}{\ini_\tau(\calJ^t_{K_m,P_n})}\right)\ge\reg\left(\frac{S}{\calJ^t_{K_m,P_n}}\right)\ge\reg\left(\frac{S}{\calJ^t_{K_2,P_n}}\right)=2(t-1)+(n-1) .
    \]
    Thus, it suffices to prove that $\reg({S}/{\ini_\tau(\calJ^t_{K_m,P_n})})\le 2(t-1)+(n-1)$.  
    To achieve this goal,
    we note that $\ini_\tau(\calJ^t_{K_m,P_n})= (\ini_\tau(\calJ_{K_m,P_n}))^t =I(H)^t$ by \Cref{lem:sagbi_path}.  Furthermore, it follows from \Cref{lem:reg_cochord,lem:cochord} that $\reg(S/I(H)^t)\le 2(t-1)+(n-1)$. So we are done. 
\end{proof}

Next, we determine the limit of the depth of powers of $\calJ_{K_m,P_n}$.

\begin{Theorem}
    \label{limitdepth}
    The following results hold under the assumptions in \Cref{set:path}:
    \begin{enumerate}[a]
        \item 
            \[
                \depth\left(\frac{S}{\calJ_{K_m,P_n}}\right)=\depth\left(\frac{S}{\ini_\tau(\calJ_{K_m,P_n})}\right)=n+(m-1);
            \]
        \item \label{limitdepth_c} 
            for each $t\ge 1$, we have
            \[
                \depth\left(\frac{S}{\ini_\tau(\calJ^t_{K_m,P_n})}\right)\ge \depth\left(\frac{S}{\ini_\tau(\calJ^{t+1}_{K_m,P_n})}\right);
            \]
        \item 
            \[
                \lim\limits_{t\to\infty} \depth\left(\frac{S}{\calJ^t_{K_m,P_n}}\right)= \lim\limits_{t\to\infty} \depth\left(\frac{S}{\ini_\tau(\calJ^t_{K_m,P_n})}\right)=
                \begin{cases}
                    3, & \text{if $m\ge 3$,}\\
                    n+1, & \text{if $m=2$.}
                \end{cases}
            \]
    \end{enumerate} 
\end{Theorem}
\begin{proof} 
    \begin{enumerate}[a]
        \item 
            Recall that a chordal graph is said to be a \emph{generalized block graph} if it satisfies: for any three maximal cliques $F_i$, $F_j$, and $F_k$, if $F_i\cap F_j \cap F_k\ne \emptyset$, then $F_i\cap F_j =F_i\cap F_k=F_j\cap F_k$. Since the path graph $P_n$ is clearly a generalized block graph whose clique number is $2$, this part follows from \cite[Theorem 3.3]{MR4233116}.

        \item  By \Cref{lem:sagbi_path}, we get that $\ini_\tau(\calJ_{K_m,P_n}^t)=(\ini_\tau(\calJ_{K_m,P_n}))^t=I(H)^t$
            for all $t\ge 1$, where $H$ is the bipartite graph constructed in \Cref{set:path}. Therefore, it suffices to prove that $\depth\left(S/I(H)^t\right)\ge \depth\left(S/I(H)^{t+1}\right)$ for all $t\ge 1$.  Since $H$ is a bipartite graph, we have $I(H)^t=I(H)^{(t)}$ for any $t\ge 1$ by \cite[Theorem 14.3.6 and Corollary 14.3.15]{MR3362802}, where $I(H)^{(t)}$ is the $t$-th symbolic power of $I(H)$. In addition, note that $x_{m-1,1}$ and $x_{2,n}$ are two leaves of $H$, we obtain that $\pd(S/I(H)^{(t+1)})\ge \pd(S/I(H)^{(t)})$ by \cite[Theorem 5.2]{MR3798623}. The desired statement then follows from the Auslander--Buchsbaum formula; cf, for example, \cite[Corollary A.4.3]{MR2724673}.

        \item For all $t\ge 1$, by \cite [Theorem 3.3.4]{MR2724673}, we have 
            \begin{equation*}
                \depth\left(\frac{S}{\calJ^t_{K_m,P_n}}\right)\ge  \depth\left(\frac{S}{\ini_\tau(\calJ^t_{K_m,P_n})}\right).
            \end{equation*}
            It follows from the previous part \ref{limitdepth_c}, \cite[Proposition 3.3]{MR696134}, \Cref{lem:sagbi_path}, and \Cref{cor:CM}\ref{cor:CM-d} that
            \begin{align}
                \lim\limits_{t\to\infty} \depth\left(\frac{S}{\calJ^t_{K_m,P_n}}\right)& \ge   \lim\limits_{t\to\infty} \depth\left(\frac{S}{\ini_\tau(\calJ^t_{K_m,P_n})}\right)=\inf_{t\ge 1} \depth\left(\frac{S}{\ini_\tau(\calJ^t_{K_m,P_n})}\right)
                \notag\\
                & =  
                \inf_{t\ge 1} \depth\left(\frac{S}{(\ini_\tau(\calJ_{K_m,P_n}))^t}\right) =
                \dim(S)-\ell(\ini_\tau(\calJ_{K_m,P_n}))   \label{eqn:limit_1}
            \end{align}
            where  $\ell(\ini_\tau(\calJ_{K_m,P_n}))$ is the analytic spread of $\ini_\tau(\calJ_{K_m,P_n})$. 

            At the same time, it follows from that \cite [Theorem 1.2]{MR2163482} that $\lim\limits_{t\to\infty} \depth\left(S/\calJ^t_{K_m,P_n}\right)$ exists and satisfies
            \begin{equation}
                \lim\limits_{t\to\infty} \depth\left(\frac{S}{\calJ^t_{K_m,P_n}}\right)\le \dim(S)-\ell(\calJ_{K_m,P_n}).
                \label{eqn:limit_2}
            \end{equation}
            If we combine \Cref{eqn:limit_1,eqn:limit_2} together, it remains to apply \Cref{cor:CM}\ref{cor:CM_analytic_spread}.        
            \qedhere
    \end{enumerate}
\end{proof} 

We end this section with the study of the symbolic powers of $\calJ_{K_m,P_n}$. 
Let $I$ be an ideal of $S$, and suppose that $\Ass(I)$ is the set of associated prime ideals of $I$. For any integer $t\ge 1$, the \emph{$t$-th symbolic power} of $I$ is defined by
\[
    I^{(t)} \coloneqq \bigcap\limits_{\frakp\in \Ass(I)}(I^t S_\frakp\cap S).
\]
In most cases, symbolic powers are not identical to the ordinary powers. However, Ene and Herzog proved in \cite{MR4143239} that if $G$ is a closed graph and $J_G$ is its binomial edge ideal, then $J_G^{(t)}=J_G^t$ for all $t\ge 1$. Recall that $G$ is said to be \emph{closed} if for all edges $\{i,j\}$ and $\{k,l\}$ of $G$ with $i<j$ and $k<l$, one has $\{j,l\}\in E(G)$ if $i=k$, and $\{i,k\} \in E(G)$ if $j=l$. Path graphs are the simplest closed graphs. In the remaining of this section, we show that the symbolic powers of the \emph{generalized} binomial edge ideal of a path graph coincide with the ordinary powers.

\begin{Theorem}
    \label{thm:symbolic_power} 
    Let $t$ be a positive integer and 
    $\calJ^{(t)}_{K_m,P_n}$ be the $t$-th symbolic power of $\calJ_{K_m,P_n}$. Then, we have  $\calJ^{(t)}_{K_m,P_n}=\calJ^{t}_{K_m,P_n}$.
\end{Theorem}

The proof of this result is involved. We need some preparations.

\begin{Lemma}
    \label{lem:symbolic_power_conj_2}
    Let $I=I_2(\bdX)$ be the ideal in $S=\KK[\bdX]$, which is generated by all $2$-minors of the $m\times n$ matrix $\bdX$. Then, one has $\ini_\tau(I^{(t)})=(\ini_\tau(I))^{(t)}$ for every integer $t\ge 1$.
\end{Lemma}
\begin{proof}
    Consider the graph $\widetilde{G}$ with the edge set 
    \[
        E(\widetilde{G})=\Set{\{x_{i,j},x_{i',j'}\}:1\le i<i'\le m, 1\le j<j'\le n}.
    \]
    It is clear that $\ini_\tau(I)$ is the edge ideal of $\widetilde{G}$ in $S$. To confirm the expected equality, we first show that $\widetilde{G}$ is \emph{perfect} in the sense that the \emph{chromatic number} $\chi(G_{V'})$ equals the \emph{clique number} $\omega(G_{V'})$ for every subset $V'$ of $V(\widetilde{G})$. Recall that a graph is called a \emph{comparability graph} if there exists a partial ordering of its vertex set such that two vertices are adjacent if and only if they are comparable. It is well-known that every comparability graph is perfect, see \cite[Corollary 14.5.6]{MR3362802}. At the same time, the graph $\widetilde{G}$ here is a comparability graph, where the poset structure can be taken with respect to the subscripts of the vertices of $\widetilde{G}$: $x_{i,j}\prec x_{i',j'}$ if and only if both $i<i'$ and $j<j'$. Thus, $\widetilde{G}$ is a perfect graph.

    By looking at the graded components, it follows from \cite[Corollary 13.7.2]{MR3362802} that $(\ini_\tau(I))^{(t)}$ is generated by monomials of the form $u_1u_2\dots u_s$ such that
    \[
        u_k=x_{i_{k,1},j_{k,1}}x_{i_{k,2},j_{k,2}}\cdots x_{i_{k,r_k},j_{k,r_k}}
    \]
    with $1\le i_{k,1}<i_{k,2}<\cdots<i_{k,r_k}\le m$, $1\le j_{k,1}<j_{k,2}<\cdots<j_{k,r_k}\le n$, and $\sum_{k=1}^s (r_k-1) = t$. 

    At the same time, it follows from \cite[Theorem 4.3.6]{zbMATH07539232} that $\ini_\tau(I^{(t)})$
    is generated by the monomials $\ini_\tau(\Delta)$, where $\Delta$ is a product of minors with $\gamma_2(\Delta) = t$ and no factor of size $<2$. These monomials are precisely those described in the previous paragraph. Thus, the proof is completed.
\end{proof}

Let $G$ be a simple graph and let $c(G)$ denote the number of connected components of $G$. A vertex $v$ is called a \emph{cut vertex} of $G$ if $c(G)< c(G\setminus v)$. Let $A$ be a subset of $V(G)$. By abuse of notation, we also let $c(A)$ denote the number of connected components of $G\setminus A$. If $v$ is a cut vertex of the induced subgraph $G\setminus (A\setminus \{v\})$ for any $v\in A$, then we say that $A$ has the \emph{cut point property}. Set $ \calC(G) \coloneqq \{\emptyset\}\cup\Set{A: A \text{\ has the cut point property}}$.

Now, let $G$ be a simple graph on the vertex set $[n]$. For each subset $A$ of $[n]$, we introduce the ideal
\[
    P_A(K_m, G)\coloneqq (x_{ij}: (i,j)\in [m]\times
    A)+\calJ_{K_m,\widetilde{G_1}}+\cdots+\calJ_{K_m,\widetilde{G_{c(A)}}}
\]
in $S$, where $G_1,\ldots, G_{c(A)}$ are the connected components of $G\setminus A$. 
It is well-known that 
\[
    \calJ_{K_m,G}=\bigcap_{A\in \calC(G)}P_A(K_m, G)
\]
is the minimal primary decomposition of the radical ideal $\calJ_{K_m,G}$; see \cite[Theorem 7]{MR3011436}.

\begin{Lemma}
    \label{lem:symbolic_power_conj_1}
    Let $A\subset [n]$. Then $\ini_\tau(P_A(K_m,G)^{(t)})=(\ini_\tau(P_A(K_m,G)))^{(t)}$ for any integer $t\ge 1$.
\end{Lemma}
\begin{proof}
    We follow the strategy of \cite[Lemma 3.2]{MR4143239}. The only essential change is that we consider instead the symbolic Rees algebra $\calR_s(I)\coloneqq \oplus_{k\ge 0}I^{(k)}T^k$ of an ideal $I$ in $S$, which is a graded subalgebra of $S[T]$.

    For simplicity, we write $P$ instead of $P_A(K_m,G)$, $c$ instead of $c(A)$, and $J_k$ instead of $\calJ_{K_m,\widetilde{G_{k}}}$ for $1\le k\le c$. Since the sets of variables $\{x_{i,j}:i\in [m],j\in V(\widetilde{G_k})\}$ as well as the set $\Set{x_{i,j}:i\in [m],j\in A}$ are pairwise disjoint, we have
    \begin{equation}
        \ini_\tau(P)=(x_{i,j}:i\in [m],j\in A)+\ini_\tau(J_1)+\cdots+\ini_\tau(J_c).
        \label{eqn:initial_prime}
    \end{equation}
    It follows from the same pairwise disjointness and \cite[Theorem 3.4]{MR4074049} that
    \begin{align}
        \calR_s(P)&=\calR_s((x_{i,j}:i\in [m],j\in A))\otimes_{\KK}(\otimes_{k=1}^c \calR_s(J_k)),
        \label{eqn:symbolic_Rees} \\
        \intertext{and}
        \calR_s(\ini_\tau(P)) & = \calR_s((x_{i,j}:i\in [m],j\in A))\otimes_{\KK}(\otimes_{k=1}^c \calR_s(\ini_\tau(J_k))) \notag\\
        & = \calR_s((x_{i,j}:i\in [m],j\in A))\otimes_{\KK}(\otimes_{k=1}^c \ini_{\tau'}(\calR_s(J_k))),
        \label{eqn:symbolic_Rees_init}
    \end{align}
    where the last equality is essentially due to \Cref{lem:symbolic_power_conj_2}. 
    Since it is clear that 
    $\text{RHS of \eqref{eqn:symbolic_Rees_init}} \subseteq  \ini_{\tau'} (\text{RHS of \eqref{eqn:symbolic_Rees}})$,
    we have 
    \begin{equation}
        \calR_s(\ini_\tau(P))\subseteq \ini_{\tau'}(\calR_s(P)).
        \label{eqn:containment_of_symbolic_powers}
    \end{equation}
    On the other hand, it is well-known that $\calR_s(P)$ and $\ini_{\tau'}(\calR_s(P))$ have the same Hilbert functions. Since $\calR_s(J_k)$ and $\ini_{\tau'}(\calR_s(J_k))$ have the same Hilbert functions for every $k$, it follows from \cref{eqn:symbolic_Rees,eqn:symbolic_Rees_init} that $\calR_s(\ini_\tau(P))$ and $\ini_{\tau'}(\calR_s(P))$ have the same Hilbert functions. Therefore, we have $\calR_s(\ini_\tau(P))=\ini_{\tau'}(\calR_s(P))$ from \cref{eqn:containment_of_symbolic_powers}. At the level of graded components, we obtain that $\ini_\tau(P^{(t)})=(\ini_\tau(P))^{(t)}$ for every $t$.
\end{proof}

\begin{Lemma}
    \label{lem:prime_decomp_initial_gen_block}
    Let $G$ be a generalized block graph. Then, 
    \begin{equation*}
        \ini_\tau(\calJ_{K_m,G})=\bigcap_{\frakp\in\Ass(\calJ_{K_m,G})} \ini_\tau(\frakp).
    \end{equation*}
\end{Lemma}
\begin{proof}
    Without loss of generality, we assume that $G$ is connected. Let $r$ be the number of maximal cliques of $G$. When $r=1$, $G$ is a complete graph. Since $\calJ_{K_m,G}$ is a prime ideal in this case, the expected result is trivial. When $r>1$, we apply the following observation from the proof of \cite[Theorem 3.3]{MR4233116}: There exists a leaf order, say, $F_1,\dots,F_r$, on the clique complex $\Delta(G)$ of $G$.
    Let $F_{t_1},\dots,F_{t_q}$ be the branches of $F_r$. Then, the intersection of any pair of facets from $F_{t_1},\dots,F_{t_q}, F_r$ is the same set of vertices, say, $A$. Now, $\calJ_{K_m,G}=J_1\cap J_2$, where $J_1\coloneqq \bigcap_{B\in \calC(G), A\cap B=\emptyset}P_B(G)$ and $J_2\coloneqq \bigcap_{B\in \calC(G),A\subseteq B}P_B(G)$. Note that $J_1=\calJ_{K_m,G'}$ where $G'$ is obtained from $G$ by replacing the cliques $F_{t_1},\dots,F_{t_q},F_r$ by the clique on the vertex set $F_r\cup (\cup_{k=1}^q F_{t_k})$. At the same time, $J_2=(x_{i,j}:(i,j)\in [m]\times A)+\calJ_{K_m,G''}$, where $G''$ is the restriction of $G$ to the vertex set $V(G)\setminus A$. Obviously, $G'$ and $G''$ are generalized block graphs with fewer maximal cliques. Since we have $\ini_\tau(\calJ_{K_m,G})=\ini_\tau(J_1)\cap \ini_\tau(J_2)$ from the proof of \cite[Theorem 3.3(c)]{MR4233116} in the $r>1$ case, we are done by induction on $r$.
\end{proof}

\begin{Lemma}
    \label{lem:MR4143239_3_1}
    Let  $t$ be a positive integer. Suppose that $I$ is an ideal in $S$ and $\sigma$ is a term order such that $\ini_{\sigma}(I)$ is a squarefree monomial ideal. It follows that $I$ is a radical ideal and $I=\bigcap_{\frakp\in \Min(I)} \frakp$.
    Suppose that the following conditions are satisfied:
    \begin{enumerate}[a]
        \item $\ini_{\sigma}(\frakp)$ is a squarefree monomial ideal for each $\frakp\in \Min(I)$,
        \item $(\ini_{\sigma}(I))^{(t)}=(\ini_{\sigma}(I))^{t}$,
        \item $\ini_{\sigma}(\frakp^{(t)})=(\ini_{\sigma}(\frakp))^{(t)}$ for each $\frakp\in \Min(I)$,
        \item $\ini_{\sigma}(I)=\bigcap_{\frakp\in \Min(I)}\ini_{\sigma}(\frakp)$.
    \end{enumerate}
    Then  $I^{(t)}=I^t$.
\end{Lemma}
\begin{proof}
    Since $\ini_{\sigma}(I)$ is a squarefree monomial ideal, we have
    \[
        (\ini_{\sigma}(I))^{(t)}=\bigcap_{\frakP\in \Min(\ini_{\sigma}(I))} \frakP^t.
    \]
    Likewise, we have 
    \[
        (\ini_{\sigma}(\frakp))^{(t)}=\bigcap_{\frakP\in \Min(\ini_{\sigma}(\frakp))} \frakP^t
    \]
    for each $\frakp\in \Min(I)$.
    Therefore, we have
    $\left(\bigcap_{\frakp\in \Min(I)} \ini_{\sigma}(\frakp)\right)^{(t)} {=} \bigcap_{\frakp\in \Min(I)} (\ini_{\sigma}(\frakp))^{(t)}$.
    It follows from the remaining assumptions that
    \begin{align*}
        \ini_{\sigma}(I^t) &\supseteq (\ini_{\sigma}(I))^t = (\ini_{\sigma}(I))^{(t)}
        = \left(\bigcap_{\frakp\in \Min(I)} \ini_{\sigma}(\frakp)\right)^{(t)} {=} \bigcap_{\frakp\in \Min(I)} (\ini_{\sigma}(\frakp))^{(t)} \\
        &= \bigcap_{\frakp\in \Min(I)} \ini_{\sigma}(\frakp^{(t)}) \supseteq \ini_{\sigma} \left(\bigcap_{\frakp\in \Min(I)} \frakp^{(t)}\right) = \ini_{\sigma}(I^{(t)}) \supseteq \ini_{\sigma}(I^{t}).
    \end{align*}
    Therefore, we obtain that $\ini_{\sigma}(I^{(t)})=\ini_{\sigma}(I^t)$. Since $I^t\subseteq I^{(t)}$, we get $I^t=I^{(t)}$.
\end{proof}

\begin{proof}
    [Proof of \Cref{thm:symbolic_power}.]
    For each $\frakp\in \Min(\calJ_{K_m,P_n})$, we know $\ini_\tau(\frakp)$ is squarefree from \cref{eqn:initial_prime}. Furthermore, notice that $\ini_\tau(\calJ_{K_m,P_n})=I(H)$ where $H$ is a bipartite graph. It follows from \cite[Corollary 13.3.6]{MR3362802} that $(\ini_\tau(\calJ_{K_m,P_n}))^t=(\ini_\tau(\calJ_{K_m,P_n}))^{(t)}$. It remains to apply \Cref{lem:prime_decomp_initial_gen_block,lem:symbolic_power_conj_1,lem:MR4143239_3_1}.
\end{proof}

\begin{Remark}
    \Cref{lem:MR4143239_3_1} is a modification of \cite[Lemma 3.1]{MR4143239}. Note that we cannot use \cite[Lemma 3.1]{MR4143239} directly to prove \Cref{thm:symbolic_power}, since the condition (ii)(a) of \cite[Lemma 3.1]{MR4143239} is not satisfied by $\calJ_{K_m,P_n}$ when $m=n=3$. In fact, in this case, there is a prime ideal $\frakp$ associated with $\calJ_{K_3,P_3}$ such that $\frakp^2\ne \frakp^{(2)}$. This prime ideal is $P_{\emptyset}(K_3, P_3)=I_2(\bdX)$, where $\bdX$ is the $3\times 3$ generic matrix. It follows from \cite[Theorem 4.3.6]{zbMATH07539232} that
    $I_2^{(2)}(\bdX)=I_2^2(\bdX)+I_3(\bdX)$.
\end{Remark}

\section{Blowup algebras}
\label{sec:blowup}

In this section, we will use algebraic properties of the initial algebras and the Sagbi basis theory to study the regularities of the blowup algebras of the ideal $\calJ_{K_m,P_n}$. Our approach will involve a combination of combinatorial optimization techniques to analyze the related algebraic invariants.

\begin{Lemma}
    \label{lem:match_number}
    For the bipartite graph $H$ introduced in \Cref{set:path}, we have
    \[
        \match(H)= (m-1)\floor{\frac{n}{2}}+\floor{\frac{n-1}{2}},
    \]
    where $\lfloor\frac{n}{2}\rfloor$ is the largest integer $\le \frac{n}{2}$.
\end{Lemma}
\begin{proof}
    We have mentioned in \Cref{rmk:graph_H} that $x_{m,1}$ and $x_{1,n}$ are two isolated vertices of $H$. Let $H'$ be the induced subgraph $H\setminus \{x_{m,1},x_{1,n}\}$. This is a bipartite graph with a bipartition $V_1'\sqcup V_2'$, where $V_i'=V_i\cap V(H')$ for $i=1,2$. Notice that we have a complete matching from $V_2'$ to $V_1'$, given by
    \begin{equation}
        \{\{x_{i,j},x_{i+1,j+1}\}: i\in [m-1],\text{$j$ is odd}\} \cup \{\{x_{1,j},x_{m,j+1}\}:\text{$j$ is even}\};
        \label{eqn:match}
    \end{equation}
    see also \Cref{Fig:7}. Therefore, we have the desired matching number, by counting the number of edges in \cref{eqn:match}.
\end{proof}

\begin{figure}
    \begin{minipage}{0.48\textwidth}
        \centering
        \begin{tikzpicture}[thick, scale=0.8, every node/.style={scale=0.98}]
            \draw[very thick](1,4)--(2,3);
            \draw[very thick](1,3)--(2,2);
            \draw[very thick](1,2)--(2,1);
            \draw[very thick](2,4)--(3,1);
            \draw[very thick](3,4)--(4,3);
            \draw[very thick](3,3)--(4,2);
            \draw[very thick](3,2)--(4,1);
            \draw[very thick](4,4)--(5,1);
            \draw[very thick](5,4)--(6,3);
            \draw[very thick](5,3)--(6,2);
            \draw[very thick](5,2)--(6,1);

            \shade [shading=ball, ball color=black] (1,1) circle (.07) node [left] {\scriptsize $x_{4,1}$};
            \shade [shading=ball, ball color=black] (2,1) circle (.07);
            \shade [shading=ball, ball color=black] (3,1) circle (.07);
            \shade [shading=ball, ball color=black] (4,1) circle (.07);
            \shade [shading=ball, ball color=black] (5,1) circle (.07);
            \shade [shading=ball, ball color=black] (6,1) circle (.07);
            \shade [shading=ball, ball color=black] (1,2) circle (.07);
            \shade [shading=ball, ball color=black] (2,2) circle (.07);
            \shade [shading=ball, ball color=black] (3,2) circle (.07);
            \shade [shading=ball, ball color=black] (4,2) circle (.07);
            \shade [shading=ball, ball color=black] (5,2) circle (.07);
            \shade [shading=ball, ball color=black] (6,2) circle (.07);
            \shade [shading=ball, ball color=black] (1,3) circle (.07);
            \shade [shading=ball, ball color=black] (2,3) circle (.07);
            \shade [shading=ball, ball color=black] (3,3) circle (.07);
            \shade [shading=ball, ball color=black] (4,3) circle (.07);
            \shade [shading=ball, ball color=black] (5,3) circle (.07);
            \shade [shading=ball, ball color=black] (6,3) circle (.07);
            \shade [shading=ball, ball color=black] (1,4) circle (.07);
            \shade [shading=ball, ball color=black] (2,4) circle (.07);
            \shade [shading=ball, ball color=black] (3,4) circle (.07);
            \shade [shading=ball, ball color=black] (4,4) circle (.07);
            \shade [shading=ball, ball color=black] (5,4) circle (.07);
            \shade [shading=ball, ball color=black] (6,4) circle (.07)  node [right] {\scriptsize$x_{1,6}$};
            \node at (6,4) {$\times$};
            \node at (1,1) {$\times$};
        \end{tikzpicture}
        \subcaption*{$(m,n)=(4,6)$}
    \end{minipage}\hfill
    \begin{minipage}{0.48\textwidth}
        \centering
        \begin{tikzpicture}[thick, scale=0.8, every node/.style={scale=0.98}]]
            \draw[very thick](1,4)--(2,3);
            \draw[very thick](1,3)--(2,2);
            \draw[very thick](1,2)--(2,1);
            \draw[very thick](2,4)--(3,1);
            \draw[very thick](3,4)--(4,3);
            \draw[very thick](3,3)--(4,2);
            \draw[very thick](3,2)--(4,1);
            \draw[very thick](4,4)--(5,1);
            \draw[very thick](5,4)--(6,3);
            \draw[very thick](5,3)--(6,2);
            \draw[very thick](5,2)--(6,1);
            \draw[very thick](6,4)--(7,1);

            \shade [shading=ball, ball color=black] (1,1) circle (.07) node [left] {\scriptsize$x_{4,1}$};
            \shade [shading=ball, ball color=black] (2,1) circle (.07);
            \shade [shading=ball, ball color=black] (3,1) circle (.07);
            \shade [shading=ball, ball color=black] (4,1) circle (.07);
            \shade [shading=ball, ball color=black] (5,1) circle (.07);
            \shade [shading=ball, ball color=black] (6,1) circle (.07);
            \shade [shading=ball, ball color=black] (7,1) circle (.07);
            \shade [shading=ball, ball color=black] (1,2) circle (.07);
            \shade [shading=ball, ball color=black] (2,2) circle (.07);
            \shade [shading=ball, ball color=black] (3,2) circle (.07);
            \shade [shading=ball, ball color=black] (4,2) circle (.07);
            \shade [shading=ball, ball color=black] (5,2) circle (.07);
            \shade [shading=ball, ball color=black] (6,2) circle (.07);
            \shade [shading=ball, ball color=black] (7,2) circle (.07);
            \shade [shading=ball, ball color=black] (1,3) circle (.07);
            \shade [shading=ball, ball color=black] (2,3) circle (.07);
            \shade [shading=ball, ball color=black] (3,3) circle (.07);
            \shade [shading=ball, ball color=black] (4,3) circle (.07);
            \shade [shading=ball, ball color=black] (5,3) circle (.07);
            \shade [shading=ball, ball color=black] (6,3) circle (.07);
            \shade [shading=ball, ball color=black] (7,3) circle (.07);
            \shade [shading=ball, ball color=black] (1,4) circle (.07);
            \shade [shading=ball, ball color=black] (2,4) circle (.07);
            \shade [shading=ball, ball color=black] (3,4) circle (.07);
            \shade [shading=ball, ball color=black] (4,4) circle (.07);
            \shade [shading=ball, ball color=black] (5,4) circle (.07);
            \shade [shading=ball, ball color=black] (6,4) circle (.07);
            \shade [shading=ball, ball color=black] (7,4) circle (.07) node [right] {\scriptsize$x_{1,7}$};
            \node at (7,4) {$\times$};
            \node at (1,1) {$\times$};
        \end{tikzpicture}
        \subcaption*{$(m,n)=(4,7)$}
    \end{minipage}
    \caption{A complete matching of $H$}
    \label{Fig:7}
\end{figure}
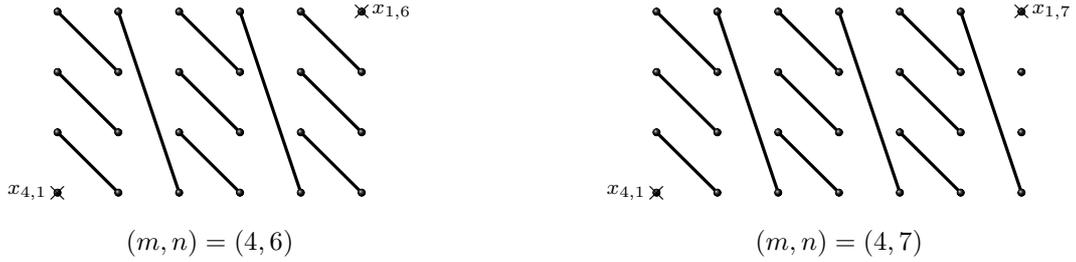

\begin{Theorem}
    \label{regRees}
    Under the assumptions in   \Cref{set:path}, we have
    \[
        \reg(\calR(\calJ_{K_m,P_n}))=\reg(\calR(\ini_\tau(\calJ_{K_m,P_n})))=
        (m-1)\floor{\frac{n}{2}}+\floor{\frac{n-1}{2}}.
    \]
\end{Theorem}
\begin{proof} 
    Note that the ideal $\calJ_{K_m,P_n}$ satisfies the following conditions:
    \begin{enumerate}[a]
        \item the natural generators of $\calJ_{K_m,P_n}$ form a Gr\"obner basis with respect to the term order $\tau$ by \Cref{rmk:graph_H};
        \item for each $t\ge 1$,  $\ini_\tau(\calJ^t_{K_m,P_n})= (\ini_\tau(\calJ_{K_m,P_n}))^t$ by \Cref{lem:sagbi_path};
        \item the Rees algebras $\calR(\ini_\tau(\calJ_{K_m,P_n}))$ is Cohen--Macaulay  by \Cref{cor:CM}. 
    \end{enumerate}
    It follows from \cite[Theorem 3.2]{MR4405525} that
    $\reg(\calR(\calJ_{K_m,P_n}))=\reg(\calR(\ini_\tau(\calJ_{K_m,P_n})))$. 

    We have seen that $\ini_\tau(\calJ_{K_m,P_n})=I(H)$ is the edge ideal of the bipartite graph $H$ constructed in \Cref{set:path}. It follows from \cite[Theorem 4.2]{MR3884545}, or implicitly from \cite[Theorems 7.1.8 and 14.3.55]{MR3362802}, that $\reg(\calR(I(H)))=\match(H)$. The only remaining step is to apply
    \Cref{lem:match_number}.
\end{proof} 

In the following, we will use the $a$-invariant of the Cohen--Macaulay special fiber ring $\calF(\calJ_{K_m,P_n})$ to compute its regularity; see \cite[Theorem 6.4.1]{MR3362802}. Recall that if $A$ is a standard graded $\KK$-algebra, the \emph{$a$-invariant} of $A$, denoted by $a(A)$, is the degree, as a rational function, of the Hilbert series of $A$.

\begin{Remark}
    Notice that $\ini_{\tau}(\calF(\calJ_{K_m,P_n}))=\KK[H]$ for the graph $H$ introduced in \Cref{set:path}.
    When $m=2$ or $3$, the graph $H$ is acyclic; see also \Cref{rmk:graph_H}. It follows from \Cref{lem:defining_ideal} that the presentation ideal of $\KK[H]$ is the zero ideal. In other words, $\calF(\ini_\tau(\calJ_{K_m,P_n}))$ is a polynomial ring in these cases. In particular, its regularity is zero. Since $\ini_\tau \calF(\calJ_{K_m,P_n})$ and $\calF(\calJ_{K_m,P_n})$ are Cohen--Macaulay and have the same Hilbert series, the regularity of the special fiber ring $\calF(\calJ_{K_m,P_n})$ is also zero.
\end{Remark}

By the previous observation, we will focus on the case when $m\ge 4$. The technical computation of the $a$-invariant of $\calF(\calJ_{K_m,P_n})$ is given in \Cref{lem:a_inv}. Before starting the proof, let us briefly introduce its strategy, which consists of two different combinatorial approaches. 

The first approach deals with directed graphs, as in 
\Cref{lem:a_inv_directed_cut}. 

\begin{Definition}
    Let $D$ be a directed graph with the vertex set $V(D)$ and the edge set $E(D)$. For every subset $A\subseteq V(D)$, define $\delta^+(A) \coloneqq \Set{ (z, w) \in E(D): z \in A, w \notin A}$ to be the set of edges leaving the vertex set $A$ and $\delta^-(A)$ to be the set of edges entering the vertex set $A$. The edge set $\delta^+(A)$ is a \emph{directed cut} of $D$ if $\emptyset \ne A \subsetneq V(D)$ and $\delta^-(A) = \emptyset$. 
\end{Definition}

\begin{Lemma}
    [{\cite[Theorem 11.3.2]{MR3362802}}]
    \label{lem:a_inv_directed_cut}
    Let $G$ be a connected bipartite graph with  bipartition $V_1\sqcup V_2$. If $G$ is regarded as a directed graph with all its arrows leaving the vertex set $V_1$, then the following two numbers are equal: 
    \begin{enumerate}[a]
        \item 
            $-a(\KK[G])$, minus the $a$-invariant of the edge subring $\KK[G]$; 
        \item 
            the maximum number of edge disjoint directed cuts of $G$. 
    \end{enumerate}
\end{Lemma}

The above lemma invites us to find edge disjoint directed cuts of $H$. A natural choice of such direct cuts gives the desired $a$-invariant that we need. However, proving that this is the value that we are looking for is a different story.

In the second approach, we start with the vector space $V\coloneqq \Mat_{m\times n}(\QQ)$ over $\QQ$. It has a canonical basis $\{\bde_{i,j}: i\in [m],j\in [n]\}$, where $\bde_{i,j}=\left(e_{k,\ell}^{i,j}\right)_{\substack{1\le k \le m\\ 1\le \ell \le n}}\in V$ with
\[
    e_{k,\ell}^{i,j}=\begin{cases}
        1, &\text{if $(k,\ell)=(i,j)$,}\\
        0, & \text{otherwise.}
    \end{cases}
\]
Let $H'$ be the graph obtained from $H$ by removing the two isolated vertices $x_{m,1}$ and $x_{1,n}$.
From this graph, we introduce a set of vectors
\[
    \calA\coloneqq \Set{\bde_{i,j}+\bde_{i',j'} : \{x_{i,j},x_{i',j'}\}\in E(H')} \subset V.
\]
The set $\RR_+ \calA$ is called the \emph{edge cone} of $H'$. At the same time, the \emph{shift polyhedron} of the edge cone of $H'$ is 
\[
    \calQ\coloneqq \conv(\Mat_{m\times n}(\ZZ) \cap \ri(\RR_{+}\calA)),
\]
where $\ri(\RR_+\calA)$ is the interior of $\RR_+\calA$ relative to the affine hull of $\RR_+\calA$; see also \cite[Corollary 11.2.4]{MR3362802}.
We will use the fact that 
\begin{equation}
    -a(\KK[H'])=\min \{\,|\bdv|/2:\bdv\in \calQ\,\} 
    \label{eqn:a_inv_b}
\end{equation}
by \cite[Theorem 11.3.1]{MR3362802}, where $|\bdv|\coloneqq \sum_{i=1}^m\sum_{j=1}^{n} v_{i,j}$ for $\bdv=(v_{i,j})\in V$.

To investigate the shift polyhedron in detail, let $V^*$ be the space of linear functions on $V$. Each $\bdF\in V^*$ defines a hyperplane $H_{\bdF}\coloneqq \{\bdv\in V: \bdF(\bdv)=0\}$ and a half-space $H_{\bdF}^+ \coloneqq \{\bdv\in V: \bdF(\bdv) \ge 0\}$. Let $\Set{\bdE_{i,j}: i\in [m],j\in [n]}\subset V^*$ be the dual basis with respect to $\{\bde_{i,j}: i\in [m],j\in [n]\}$. For simplicity, we will  represent the elements in $V^*$  as $m\times n$ matrices like $\bdF=(f_{i,j})\in V^*$. For every such function $\bdF\in V^*$ and every vector $\bdv=(v_{i,j})\in V$, we have $\bdF(\bdv)=\sum_{i,j}f_{i,j}(v_{i,j})\in \QQ$.

We have a quick remark.

\begin{Remark}
    \label{rmk:matrix_A}
    Fix a matrix $\bdA=(a_{i,j})$ in $V^{*}$, where 
    \[
        a_{i,j}=
        \begin{cases}
            0, & \text{if $(i,j)=(1,n)$ or $(m,1)$,}\\
            1, & \text{if  $j$ is otherwise  odd,}\\
            -1, & \text{if  $j$ is otherwise   even,}
        \end{cases}
    \]
    see \Cref{fig:matrix_A}. 
    \begin{figure}[tbp]
        \parbox[b]{.45\textwidth}{
            \centering
            $\displaystyle \begin{bmatrix}
                1 &-1 &1 &-1 &1 &-1& 0 \\
                1 &-1 &1 &-1 &1 &-1& 1 \\
                1 &-1 &1 &-1 &1 &-1& 1 \\
                1 &-1 &1 &-1 &1 &-1& 1 \\
                1 &-1 &1 &-1 &1 &-1& 1 \\
                0 &-1 &1 &-1 &1 &-1& 1 
            \end{bmatrix}$
            \subcaption*{$(m,n)=(6,7)$ case}
        }
        \parbox[b]{.45\textwidth}{
            \centering
            $\displaystyle \begin{bmatrix}
                1 &-1 &1 &-1 &1 &-1 &1 &0 \\
                1 &-1 &1 &-1 &1 &-1 &1 &-1 \\
                1 &-1 &1 &-1 &1 &-1 &1 &-1 \\
                1 &-1 &1 &-1 &1 &-1 &1 &-1 \\
                1 &-1 &1 &-1 &1 &-1 &1 &-1 \\
                0 &-1 &1 &-1 &1 &-1 &1 &-1
            \end{bmatrix}$
            \subcaption*{$(m,n)=(6,8)$ case}
        }
        \caption{The matrix $\bdA$ in $V^{*}$}
        \label{fig:matrix_A}
    \end{figure}
    It is easy to see that $\calA \subseteq H_{\bdA}\cap H_{\bdE_{1,n}}\cap H_{\bdE_{m,1}}$.
    Since $H'$ is a connected bipartite graph,  we get  $\dim(\KK[H'])=mn-3$  by \cite[Corollary 10.1.21]{MR3362802}. At the same time, the integral points of $\calQ$ define the canonical module of $\KK[H']$ by \cite[Proposition 11.2.1]{MR3362802}.
    Therefore, $H_{\bdA}\cap H_{\bdE_{1,n}}\cap H_{\bdE_{m,1}}$ is the minimal linear space that contains $\calQ$.
\end{Remark}

The two approaches described above will give us a lower bound and an upper bound of $-a(\KK[H'])$ respectively. Since they coincide, we obtain the exact value. Now, we carry out this strategy and start the real computation.

\begin{Proposition}
    \label{lem:a_inv}
    Suppose that $m\ge 4$ is an integer and $H$ is the bipartite graph introduced  in \Cref{set:path}. Then we have
    \[
        -a(\KK[H])  =
        \begin{cases}
            m\cdot \frac{n}{2}, &\text{if $n$ is even}, \\
            m\cdot \frac{n+1}{2}-2, & \text{if $n$ is odd}.
        \end{cases}
    \]
\end{Proposition}
\begin{proof}
    For the induced subgraph $H'$, we have $\KK[H]=\KK[H']$. Thus, we will prove 
    \begin{equation}
        -a(\KK[H'])  =
        \begin{cases}
            m\cdot \frac{n}{2}, &\text{if $n$ is even}, \\
            m\cdot \frac{n+1}{2}-2, & \text{if $n$ is odd}
        \end{cases}
        \label{eqn:a_inv}
    \end{equation}
    by considering the following two steps.
    \begin{enumerate}[A]
        \item First, we show that LHS $\ge$ RHS in \cref{eqn:a_inv}. Let $V'=V(H')$, $V'_1=\{x_{i,j}\in V':i\in [m],\text{$j$ is odd}\}$ and  $V_2'=V'\setminus V_1'$. Then $H'$ is a connected bipartite graph with  bipartition $V_1'\sqcup V_2'$. By \Cref{lem:a_inv_directed_cut}, we regard $H'$ as a directed graph with all its arrows leaving the vertex set $V_1'$. For each $u\in V_1'$, the directed cut $\delta^+(\{u\})$ is the  set of edges leaving the vertex $u$. Since
            \[
                E(H') = \bigsqcup_{u\in V_1'} \delta^+(\{u\})
            \]
            is a disjoint union, we immediately have
            \[
                -a(\KK[H'])\ge |V_1'|=
                \begin{cases}
                    m\cdot \frac{n}{2}-1, &\text{if $n$ is even}, \\
                    m\cdot \frac{n+1}{2}-2, & \text{if $n$ is odd}
                \end{cases}    
            \]
            from 
            \Cref{lem:a_inv_directed_cut}. 

            At the same time, when $n$ is even, we consider the two special vertices $u_1=x_{1,n-1}\in V_1'$ and $u_2=x_{2,n}\in V_2'$. Notice that 
            \[
                \delta^+(\{u_1,u_2\})=\Set{(u_1,x_{i,n}):3\le i\le n}
            \]
            and
            \[
                \delta^+(V'\setminus \{u_2\})=\Set{(u_1,u_2)}
            \]
            are two directed cuts. Thus, 
            \[
                \delta^+(\{u_1\})=\delta^+(\{u_1,u_2\})\sqcup \delta^+(V'\setminus \{u_2\}).
            \]
            As a result, when $n$ is even, we have additionally
            \[
                -a(\KK[H'])\ge |V_1'|+1= m\cdot \frac{n}{2}, 
            \]
            from
            \Cref{lem:a_inv_directed_cut}. In short, we have LHS $\ge$ RHS in \cref{eqn:a_inv}.

        \item Second, we prove that LHS $\le$ RHS in \cref{eqn:a_inv}. By the formula in \cref{eqn:a_inv_b}, it suffices to find a suitable $\hat{\bdu}\in \calQ$ such that $|\hat{\bdu}|$ is twice the integer in the RHS of \cref{eqn:a_inv}.  

            The candidate vector $\hat{\bdu}=(u_{i,j})$ in $V$ is given by
            \[
                u_{i,j}=
                \begin{cases}
                    0, &\text{if $(i,j)=(1,n)$ or $(m,1)$,}\\
                    2, &\text{if $(i,j)=(1,n-1)$,}\\
                    2, &\text{if $(i,j)=(m,2)$ and $n$ is even,}\\
                    m-2, &\text{if $(i,j)=(m,2)$ and $n$ is odd,}\\
                    1, & \text{otherwise;}
                \end{cases}
            \]
            see \Cref{fig:hat_u}. It is easy to verify  that $|\hat{\bdu}|$ satisfies the requirement.
            Therefore, it remains to show that $\hat{\bdu}$ belongs to $\calQ$.
            \begin{figure}[tbp]
                \parbox[b]{.45\textwidth}{
                    \centering
                    $\displaystyle \begin{bmatrix}
                        1 &1 &1 &1 &1 &2 &0 \\
                        1 &1 &1 &1 &1 &1 &1 \\
                        1 &1 &1 &1 &1 &1 &1 \\
                        1 &1 &1 &1 &1 &1 &1 \\
                        1 &1 &1 &1 &1 &1 &1 \\
                        0 &4 &1 &1 &1 &1 &1
                    \end{bmatrix}$
                    \subcaption*{$(m,n)=(6,7)$ case}
                }
                \parbox[b]{.45\textwidth}{
                    \centering
                    $\displaystyle \begin{bmatrix}
                        1 &1 &1 &1 &1 &1 &2 &0 \\
                        1 &1 &1 &1 &1 &1 &1 &1 \\
                        1 &1 &1 &1 &1 &1 &1 &1 \\
                        1 &1 &1 &1 &1 &1 &1 &1 \\
                        1 &1 &1 &1 &1 &1 &1 &1 \\
                        0 &2 &1 &1 &1 &1 &1 &1 
                    \end{bmatrix}$
                    \subcaption*{$(m,n)=(6,8)$ case}
                }
                \caption{The extremal vector $\hat{\bdu}$ in $V$}
                \label{fig:hat_u}
            \end{figure}

            First, we consider the case where $n$ is odd. We do this in two sub-steps.
            \begin{enumerate}[listparindent=1em, parsep=0pt]
                \item 
                In the first sub-step, we show that $\hat{\bdu}$ belongs to the polyhedron $\RR_+\calA$.

                    Let $\widetilde{H}$ be a subgraph of $H'$. By abuse of notation, the \emph{degree matrix} of  $\widetilde{H}$ is the $m\times n$ matrix $\bdD=(d_{i,j})$, where \[
                        \qquad
                        \qquad
                        \qquad
                        d_{i,j}=
                        \begin{cases}
                        0,& \text{if $(i,j)=(1,n)$ or $(m,1)$,}\\
                        \text{the degree of the vertex $x_{i,j}$ in $\widetilde{H}$,} &  \text{otherwise.}
                        \end{cases}
                    \]
                    We will construct subgraphs $\widetilde{H}$ of $H'$ such that the degree matrix of $\widetilde{H}$ is given by the $\hat{\bdu}$ above; we will call such subgraphs of \emph{$\hat{\bdu}$-type}. 

                    The first instance $\widetilde{H}_1$ of $\hat{\bdu}$-type can be constructed as follows. For simplicity, we will say that edges of the form $\{x_{i,j},x_{i',j+1}\}$ belong to the zone $\calZ_j$ for each 
                    $j\in  [n-1]$. For $\widetilde{H}_1$, the edges in zone $\calZ_1$ are
                    \[
                        \Set{\{x_{1,1},x_{m-1,2}\},\{x_{2,1},x_{m,2}\}, \{x_{3,1},x_{m,2}\},\dots,\{x_{m-1,1},x_{m,2}\}}.
                    \]
                    Note that it contains two long parallel edges. For $2\le j< n$, where $j$ is odd, there are only two  long parallel edges in the zone $\calZ_j$: $\{x_{1,j},x_{m-1,j+1}\}$ and $\{x_{2,j},x_{m,j+1}\}$. For $2\le j< n$ with $j$ being even, there are $m-2$ parallel slightly shorter edges in the zone $\calZ_j$: $\{x_{i,j},x_{i+2,j+1}\}$ with $1\le i\le m-2$. Finally,  we supplement the last zone $\calZ_{n-1}$ with the extra edge $\{x_{1,n-1},x_{2,n}\}$. Then, we get all the edges for the graph $\widetilde{H}_1$. At this point, the reader is invited to look at  the first graph
                    of \Cref{fig:H67}. Notice that $\hat{\bdu}=\sum_{\{x_{i,j},x_{i',j+1}\}\in E(\widetilde{H}_1)}(\bde_{i,j}+\bde_{i',j+1})$. Consequently, $\hat{\bdu}\in \RR_+\calA$. 
                    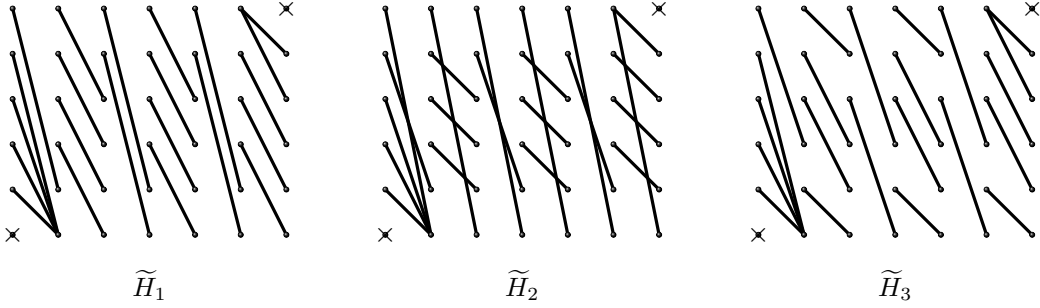
\begin{figure}[tbp]
                        \begin{minipage}{0.3\textwidth}
                            \centering
                            \begin{tikzpicture}[thick, scale=0.6, every node/.style={scale=0.98}]]
                                \draw[very thick](7,5)--(6,6);
                                \draw[very thick](7,4)--(6,6);
                                \draw[very thick](7,3)--(6,5);
                                \draw[very thick](7,2)--(6,4);
                                \draw[very thick](7,1)--(6,3);
                                \draw[very thick](5,6)--(6,2);
                                \draw[very thick](5,5)--(6,1);
                                \draw[very thick](5,4)--(4,6);
                                \draw[very thick](5,3)--(4,5);
                                \draw[very thick](5,2)--(4,4);
                                \draw[very thick](5,1)--(4,3);
                                \draw[very thick](3,6)--(4,2);
                                \draw[very thick](3,5)--(4,1);
                                \draw[very thick](3,4)--(2,6);
                                \draw[very thick](3,3)--(2,5);
                                \draw[very thick](3,2)--(2,4);
                                \draw[very thick](3,1)--(2,3);
                                \draw[very thick](1,6)--(2,2);
                                \draw[very thick](1,5)--(2,1);
                                \draw[very thick](1,4)--(2,1);
                                \draw[very thick](1,3)--(2,1);
                                \draw[very thick](1,2)--(2,1);

                                \foreach \x in {1,2,3,4,5,6,7}
                                \foreach \y in {1,2,3,4,5,6} 
                                \shade [shading=ball, ball color=black] (\x,\y) circle (.07);
                                \node at (1,1) {$\times$};
                                \node at (7,6) {$\times$};
                            \end{tikzpicture}
                            \subcaption*{$\widetilde{H}_1$}
                        \end{minipage}
                        \begin{minipage}{0.3\textwidth}
                            \centering
                            \begin{tikzpicture}[thick, scale=0.6, every node/.style={scale=0.98}]]
                                \draw[very thick](7,5)--(6,6);
                                \draw[very thick](7,4)--(6,5);
                                \draw[very thick](7,3)--(6,4);
                                \draw[very thick](7,2)--(6,3);
                                \draw[very thick](7,1)--(6,6);
                                \draw[very thick](5,5)--(6,2);
                                \draw[very thick](5,6)--(6,1);
                                \draw[very thick](5,4)--(4,5);
                                \draw[very thick](5,3)--(4,4);
                                \draw[very thick](5,2)--(4,3);
                                \draw[very thick](5,1)--(4,6);
                                \draw[very thick](3,5)--(4,2);
                                \draw[very thick](3,6)--(4,1);
                                \draw[very thick](3,4)--(2,5);
                                \draw[very thick](3,3)--(2,4);
                                \draw[very thick](3,2)--(2,3);
                                \draw[very thick](3,1)--(2,6);
                                \draw[very thick](1,5)--(2,2);
                                \draw[very thick](1,6)--(2,1);
                                \draw[very thick](1,4)--(2,1);
                                \draw[very thick](1,3)--(2,1);
                                \draw[very thick](1,2)--(2,1);

                                \foreach \x in {1,2,3,4,5,6,7}
                                \foreach \y in {1,2,3,4,5,6} 
                                \shade [shading=ball, ball color=black] (\x,\y) circle (.07);
                                \node at (1,1) {$\times$};
                                \node at (7,6) {$\times$};
                            \end{tikzpicture}
                            \subcaption*{$\widetilde{H}_2$}
                        \end{minipage}
                        \begin{minipage}{0.3\textwidth}
                            \centering
                            \begin{tikzpicture}[thick, scale=0.6, every node/.style={scale=0.98}]]
                                \draw[very thick](7,5)--(6,6);
                                \draw[very thick](7,4)--(6,6);
                                \draw[very thick](7,3)--(6,5);
                                \draw[very thick](7,2)--(6,4);
                                \draw[very thick](7,1)--(6,2);
                                \draw[very thick](5,6)--(6,3);
                                \draw[very thick](5,4)--(6,1);
                                \draw[very thick](5,5)--(4,6);
                                \draw[very thick](5,3)--(4,5);
                                \draw[very thick](5,2)--(4,4);
                                \draw[very thick](5,1)--(4,2);
                                \draw[very thick](3,6)--(4,3);
                                \draw[very thick](3,4)--(4,1);
                                \draw[very thick](3,5)--(2,6);
                                \draw[very thick](3,3)--(2,5);
                                \draw[very thick](3,2)--(2,4);
                                \draw[very thick](3,1)--(2,2);
                                \draw[very thick](1,6)--(2,3);
                                \draw[very thick](1,5)--(2,1);
                                \draw[very thick](1,4)--(2,1);
                                \draw[very thick](1,3)--(2,1);
                                \draw[very thick](1,2)--(2,1);

                                \foreach \x in {1,2,3,4,5,6,7}
                                \foreach \y in {1,2,3,4,5,6} 
                                \shade [shading=ball, ball color=black] (\x,\y) circle (.07);
                                \node at (1,1) {$\times$};
                                \node at (7,6) {$\times$};
                            \end{tikzpicture}
                            \subcaption*{$\widetilde{H}_3$}
                        \end{minipage}
                        \caption{Three subgraphs of $\hat{\bdu}$-type in the case $(m,n)=(6,7)$}
                        \label{fig:H67}
                    \end{figure}

                \item Next, we show that $\hat{\bdu}\in \ri(\RR_{+}\calA)$. For this purpose, we show that for any $\bdF\in V^{*}$ such that $H_{\bdF}$ is a supporting hyperplane of $\RR_+\calA$ and $\hat{\bdu}\in H_{\bdF}$,  we have $\calQ\subseteq H_{\bdF}$ (and equivalently, $\RR_+\calA\subseteq H_{\bdF}$). Without loss of generality, we can assume that $\bdF$ is represented by the matrix $(f_{i,j})$ with $f_{1,n}=f_{m,1}=0$. Whence, it remains to prove that $\bdF$ is a multiple of the matrix $\bdA$, which was defined earlier in \Cref{rmk:matrix_A}.

                    To prove this, we still use the $\hat{\bdu}$-type subgraphs. Let $\widetilde{H}$ be such a subgraph.
                    Since $\hat{\bdu}=\sum_{\{x_{i,j},x_{i',j+1}\}\in E(\widetilde{H})}(\bde_{i,j}+\bde_{i',j+1})\in H_{\bdF}$, we have $\sum_{\{x_{i,j},x_{i',j+1}\}\in E(\widetilde{H})}(f_{i,j}+f_{i',j+1})=0$. 
                    On the other hand, if $\{x_{i,j},x_{i',j+1}\}\in E(\widetilde{H})$, then $\bde_{i,j}+\bde_{i',j+1}\in \calA$. Since $H_{\bdF}$ is a supporting hyperplane, we have $f_{i,j}+f_{i',j+1}\ge 0$. Therefore, we have indeed $f_{i,j}=-f_{i',j+1}$. 

                    In addition to the subgraph $\widetilde{H}_1$ in the previous part, we will construct subgraphs $\widetilde{H}_2$ and $\widetilde{H}_3$ of $\hat{\bdu}$-type, such that the subgraph $\widehat{H}$ of $H'$ with edges $E(\widetilde{H}_1)\cup E(\widetilde{H}_2)\cup E(\widetilde{H}_3)$ is a connected graph. Now, $f_{i,j}=-f_{i',j+1}\in \ZZ$ whenever $\{x_{i,j},x_{i',j+1}\}\in E(\widehat{H})$. Since $\widehat{H}$ is connected, this implies that $\bdF$ is a multiple of $\bdA$.

                    The graph $\widetilde{H}_2$ is constructed from $\widetilde{H}_1$ as follows. When $j$ is odd, the zone $\calZ_j$ contains two long parallel edges in $\widetilde{H}_1$. We cross them for $\widetilde{H}_2$. When $j$ is even, the zone $\calZ_j$ contains $m-2$ parallel edges of slope $-2$ in $\widetilde{H}_1$. We make $m-3$ of them as parallel edges of slope $-1$, and the remaining one to be $\{x_{1,j},x_{m,j+1}\}$. At this point, the reader is invited to look at the second graph of \Cref{fig:H67}. Meanwhile, note that for each $j$, one has
                    \[
                        \qquad \qquad \qquad
                        \Set{x_{i,j},x_{i',j+1} : \{x_{i,j},x_{i',j+1}\}\in E(\widetilde{H}_1)}=
                        \Set{x_{i,j},x_{i',j+1}:\{x_{i,j},x_{i',j+1}\}\in E(\widetilde{H}_2)}.
                    \]
                    For later reference, we denote this vertex set as $\widetilde{V}_j$. It is crucial to observe that the edges $(E(\widetilde{H}_1)\cup E(\widetilde{H}_2))\cap \calZ_j$ define a connected subgraph over $\widetilde{V}_j$.

                    The graph $\widetilde{H}_3$ is constructed from $\widetilde{H}_1$ as follows. For $1\le j\le n-2$ with $j$ being odd, we change the edges $\{x_{1,j},x_{m-1,j+1}\}$ and $\{x_{m-2,j+1},x_{m,j+2}\}$ in $\widetilde{H}_1$ to the edges $\{x_{1,j},x_{m-2,j+1}\}$ and $\{x_{m-1,j+1},x_{m,j+2}\}$ in $\widetilde{H}_3$. For $1\le j\le n-2$ with $j$ being even, we change the edges $\{x_{1,j},x_{3,j+1}\}$ and $\{x_{2,j+1},x_{m,j+2}\}$ in $\widetilde{H}_1$ to the edges $\{x_{1,j},x_{2,j+1}\}$ and $\{x_{3,j+1},x_{m,j+2}\}$ in $\widetilde{H}_3$. At this point, the reader is invited to look at the third graph of \Cref{fig:H67}.  Note that for each $j$, the vertex set 
                    \[
                        \Set{x_{i,j},x_{i',j+1}:\{x_{i,j},x_{i',j+1}\}\in E(\widetilde{H}_3)} 
                    \]
                    intersects both $\widetilde{V}_j$ and $\widetilde{V}_{j+1}$.  This fact makes the combined subgraph $\widehat{H}$ to be connected.
            \end{enumerate}

            In summary, we have completed the proof for the case when $n$ is odd. The case when $n$ is even is analogous, so the details will be omitted. We only give the construction of the corresponding graphs $\widetilde{H}_1$, $\widetilde{H}_2$, and $\widetilde{H}_3$ for the case  $(m,n)=(6,8)$, 
            see \Cref{fig:H68}. \qedhere
    \end{enumerate}
\end{proof}
\begin{figure}
    \begin{minipage}{0.3\textwidth}
        \centering
        \begin{tikzpicture}[thick, scale=0.56, every node/.style={scale=0.98}]]
            \draw[very thick](7,6)--(8,5);

            \draw[very thick](7,6)--(8,4);
            \draw[very thick](7,5)--(8,3);
            \draw[very thick](7,4)--(8,2);
            \draw[very thick](7,3)--(8,1);
            \draw[very thick](7,2)--(6,6);
            \draw[very thick](7,1)--(6,5);

            \draw[very thick](5,6)--(6,4);
            \draw[very thick](5,5)--(6,3);
            \draw[very thick](5,4)--(6,2);
            \draw[very thick](5,3)--(6,1);
            \draw[very thick](5,2)--(4,6);
            \draw[very thick](5,1)--(4,5);

            \draw[very thick](3,6)--(4,4);
            \draw[very thick](3,5)--(4,3);
            \draw[very thick](3,4)--(4,2);
            \draw[very thick](3,3)--(4,1);
            \draw[very thick](3,2)--(2,6);
            \draw[very thick](3,1)--(2,5);

            \draw[very thick](1,6)--(2,4);
            \draw[very thick](1,5)--(2,3);
            \draw[very thick](1,4)--(2,2);
            \draw[very thick](1,3)--(2,1);
            \draw[very thick](1,2)--(2,1);

            \foreach \x in {1,2,3,4,5,6,7,8}
            \foreach \y in {1,2,3,4,5,6} 
            \shade [shading=ball, ball color=black] (\x,\y) circle (.07);
            \node at (1,1) {$\times$};
            \node at (8,6) {$\times$};
        \end{tikzpicture}
        \subcaption*{$\widetilde{H}_1$}
    \end{minipage}
    \begin{minipage}{0.3\textwidth}
        \centering
        \begin{tikzpicture}[thick, scale=0.56, every node/.style={scale=0.98}]]
            \draw[very thick](7,6)--(8,5);
            \draw[very thick](7,5)--(8,4);
            \draw[very thick](7,4)--(8,3);
            \draw[very thick](7,3)--(8,2);
            \draw[very thick](7,6)--(8,1);
            \draw[very thick](7,2)--(6,5);
            \draw[very thick](7,1)--(6,6);

            \draw[very thick](5,5)--(6,4);
            \draw[very thick](5,4)--(6,3);
            \draw[very thick](5,3)--(6,2);
            \draw[very thick](5,6)--(6,1);
            \draw[very thick](5,2)--(4,5);
            \draw[very thick](5,1)--(4,6);

            \draw[very thick](3,5)--(4,4);
            \draw[very thick](3,4)--(4,3);
            \draw[very thick](3,3)--(4,2);
            \draw[very thick](3,6)--(4,1);
            \draw[very thick](3,2)--(2,5);
            \draw[very thick](3,1)--(2,6);

            \draw[very thick](1,5)--(2,4);
            \draw[very thick](1,4)--(2,3);
            \draw[very thick](1,3)--(2,2);
            \draw[very thick](1,2)--(2,1);
            \draw[very thick](1,6)--(2,1);

            \foreach \x in {1,2,3,4,5,6,7,8}
            \foreach \y in {1,2,3,4,5,6} 
            \shade [shading=ball, ball color=black] (\x,\y) circle (.07);
            \node at (1,1) {$\times$};
            \node at (8,6) {$\times$};
        \end{tikzpicture}
        \subcaption*{$\widetilde{H}_2$}
    \end{minipage}
    \begin{minipage}{0.3\textwidth}
        \centering
        \begin{tikzpicture}[thick, scale=0.56, every node/.style={scale=0.98}]]
            \draw[very thick](7,6)--(8,5);

            \draw[very thick](7,6)--(8,4);
            \draw[very thick](7,5)--(8,3);
            \draw[very thick](7,4)--(8,2);
            \draw[very thick](7,2)--(8,1);
            \draw[very thick](7,3)--(6,6);
            \draw[very thick](7,1)--(6,4);

            \draw[very thick](5,6)--(6,5);
            \draw[very thick](5,5)--(6,3);
            \draw[very thick](5,4)--(6,2);
            \draw[very thick](5,2)--(6,1);
            \draw[very thick](5,3)--(4,6);
            \draw[very thick](5,1)--(4,4);

            \draw[very thick](3,6)--(4,5);
            \draw[very thick](3,5)--(4,3);
            \draw[very thick](3,4)--(4,2);
            \draw[very thick](3,2)--(4,1);
            \draw[very thick](3,3)--(2,6);
            \draw[very thick](3,1)--(2,4);

            \draw[very thick](1,6)--(2,5);
            \draw[very thick](1,5)--(2,3);
            \draw[very thick](1,4)--(2,2);
            \draw[very thick](1,3)--(2,1);
            \draw[very thick](1,2)--(2,1);

            \foreach \x in {1,2,3,4,5,6,7,8}
            \foreach \y in {1,2,3,4,5,6} 
            \shade [shading=ball, ball color=black] (\x,\y) circle (.07);
            \node at (1,1) {$\times$};
            \node at (8,6) {$\times$};
        \end{tikzpicture}
        \subcaption*{$\widetilde{H}_3$}
    \end{minipage}
    \caption{Three subgraphs of $\hat{\bdu}$-type in the case $(m,n)=(6,8)$}
    \label{fig:H68}
\end{figure}
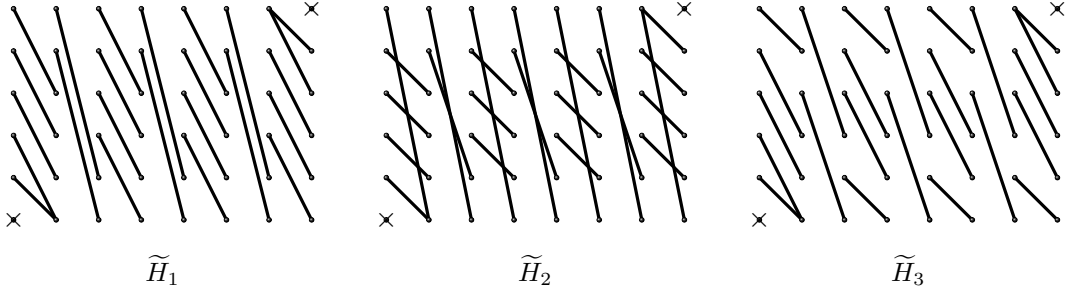

\begin{Theorem}
    \label{regfiber}
    Under the assumptions in  \Cref{set:path}, we suppose additionally that $m\ge 4$. Then, we have
    \begin{align*}
        \reg(\calF(\calJ_{K_m,P_n}))& =\reg(\calF(\ini_\tau(\calJ_{K_m,P_n}))) \\
        &= 
        \begin{cases}
            (mn-3)-(m\cdot \frac{n}{2}), &\text{if $n$ is even}, \\
            (mn-3)-(m\cdot \frac{n+1}{2}-2), & \text{if $n$ is odd}.
        \end{cases}
    \end{align*}
\end{Theorem}
\begin{proof} 
    Since $\ini_\tau(\calF(\calJ_{K_m,P_n}))=\calF(\ini_\tau(\calJ_{K_m,P_n}))$, the two special fiber rings $\calF(\calJ_{K_m,P_n})$ and $\calF(\ini_\tau(\calJ_{K_m,P_n}))$ have the same Hilbert series by \cite[Proposition 2.4]{MR1390693}. At the same time, these two special fiber rings are Cohen--Macaulay domains by 
    \Cref{cor:CM}.
    Thus, $\reg(\calF(\calJ_{K_m,P_n}))=\reg(\calF(\ini_\tau(\calJ_{K_m,P_n})))$ by 
    \cite[Corollary 2.18]{MR3838370}.

    Note that the algebras $\calF(\ini_\tau(\calJ_{K_m,P_n}))$ and $\KK[H]$ are isomorphic. Since $m\ge 4$, we have $\dim(\KK[H])=mn-3$ by \cite[Corollary 10.1.21]{MR3362802}. The desired result then follows from  \Cref{lem:a_inv} and the equality
    \[
        \reg(\KK[H])=\dim(\KK[H])+a(\KK[H]) 
    \]
    in the Cohen--Macaulay case, see \cite[Theorem 6.4.1]{MR3362802}.
\end{proof}

\section{Applications}

In this section, we will explore the generalized binomial edge ideal of a simple graph and its blowup algebras. We will compare them with the corresponding parts of an induced subgraph. Our analysis will be based on the regularity results from previous sections, which serve as natural lower bounds for the comparison.

We will start with the comparison result in a more general setting. Recall that if $G'$ is an induced subgraph of a graph $G$, then for the graded Betti numbers of the powers of the binomial edge ideals, one has $\beta_{ij}(S/J^{s}_{G'}) \le \beta_{ij}(S/J^{s}_{G})$ for all $i,j\ge 0$ and $s\ge 1$; see \cite[Proposition 3.3]{JKS}. This result for the classical binomial edge ideal can be easily generalized to the binomial edge ideal of a pair, as follows.

\begin{Setting}
    \label{set:pair}
    For $i=1,2$, let $G_i$ be a simple graph on the vertex set $[n_i]$, and let $H_i$ be its induced subgraph. Correspondingly, we have $\calJ_{G_1,G_2}$, the binomial ideal of the pair $(G_1,G_2)$ in the polynomial ring $\KK[\bdX]$, as well as $\calJ_{H_1,H_2}$, the binomial ideal of the pair $(H_1,H_2)$ in the polynomial ring $\KK[\bdY]$. If we consider $\bdX$ to be  an $n_1\times n_2$ matrix of variables, then we can naturally regard $\bdY$ as a $|V(H_1)|\times |V(H_2)|$ submatrix.
\end{Setting}

\begin{Theorem}
    \label{power}
    Under the assumptions in \Cref{set:pair}, we have 
    \[
        \beta_{ij}(\KK[\bdY]/\calJ^{s}_{H_1,H_2})\le\beta_{ij}(\KK[\bdX]/\calJ^{s}_{G_1,G_2}) 
    \]
    for all $i,j$ and $s\ge 1$. In particular, we have
    \[
        \reg(\calJ^{s}_{H_1,H_2})\le \reg(\calJ^{s}_{G_1,G_2})\qquad
        \text{and} \qquad
        \pd(\calJ^{s}_{H_1,H_2}) \le \pd(\calJ^{s}_{G_1,G_2})
    \]
    for all $s\ge 1$.
\end{Theorem}
\begin{proof}
    First, we show that $\calJ^{s}_{H_1,H_2}=\calJ^{s}_{G_1,G_2}\cap \KK[\bdY]$ for all $s\ge 1$.
    Since the natural generators of $\calJ^{s}_{H_1,H_2}$ are automatically contained in $\calJ^{s}_{G_1,G_2}$, one has $\calJ^{s}_{H_1,H_2}\subseteq\calJ^{s}_{G_1,G_2}\cap \KK[\bdY]$. For the converse  inclusion, let $g\in \calJ^{s}_{G_1,G_2}\cap \KK[\bdY]$. We can write $g$ as a finite sum
    \begin{equation*}
        g=\sum_{\substack{(e_i,f_i)\in E(G_1)\times E(G_2),\\ 1\le i\le s}}
        h_{(e_1,f_1),\ldots,(e_s,f_s)}p_{(e_1,f_1)}\cdots p_{(e_s,f_s)},
    \end{equation*}
    where $h_{(e_1,f_1),\ldots,(e_s,f_s)}\in \KK[\bdX]$.  Now, consider the $\KK$-algebra homomorphism $\pi : \KK[\bdX]\to \KK[\bdY]$ by setting 
    \[
        \pi(x_{i,j})=
        \begin{cases}
            x_{i,j},& \text{if $x_{i,j}$ is  a variable  in $\bdY$,}\\
            0,& \text{otherwise}.
        \end{cases}
    \]
    Thus,
    \[
        \pi(p_{(e,f)})=
        \begin{cases}
            p_{(e,f)}, & \text{if $(e,f)\in E(H_1)\times E(H_2)$,}\\
            0, & \text{otherwise.}
        \end{cases}
    \]
    Since $g\in \KK[\bdY]$, we have $\pi(g)=g$. Therefore, we get
    \begin{align*}
        g&=\sum_{\substack{(e_i,f_i)\in E(G_1)\times E(G_2),\\ 1\le i\le s}} 
        \pi(h_{(e_1,f_1),\ldots,(e_s,f_s)})\pi(p_{(e_1,f_1)})\cdots \pi(p_{(e_s,f_s)})\\
        &=\sum_{\substack{(e_i,f_i)\in E(H_1)\times E(H_2),\\ 1\le i\le s}}
        \pi(h_{(e_1,f_1),\ldots,(e_s,f_s)}) p_{(e_1,f_1)}\cdots p_{(e_s,f_s)}.
    \end{align*}
    Thus, $g\in \calJ^{s}_{H_1,H_2}$. This completes our proof for $\calJ^{s}_{H_1,H_2}=\calJ^{s}_{G_1,G_2}\cap \KK[\bdY]$.

    Consequently, $\KK[\bdY]/\calJ^{s}_{H_1,H_2}$ is a $\KK$-subalgebra of $\KK[\bdX]/\calJ^{s}_{G_1,G_2}$. 
    Let $\bar{\pi}:\KK[\bdX]/\calJ^{s}_{G_1,G_2} \to \KK[\bdY]/\calJ^{s}_{H_1,H_2}$ be the homomorphism induced by $\pi$. Since $\pi(\calJ^{s}_{G_1,G_2})\subseteq \calJ^{s}_{H_1,H_2}$, the map $\bar{\pi}$ is well-defined. Notice that the restriction of $\bar{\pi}$ to $\KK[\bdY]/\calJ^{s}_{H_1,H_2}$ is the identity map. Thus, $\bar{\pi}$ is surjective, and 
    $\KK[\bdY]/\calJ^{s}_{H_1,H_2}$ is an algebra retract of $\KK[\bdX]/\calJ^{s}_{G_1,G_2}$. Now, the expected inequalities follow from \cite[Corollary 2.5]{MR1789447}. 
\end{proof}

\begin{Corollary}
    Let $G$ be a simple graph
    and $G'$ be its induced subgraph. Then we have $\reg(\mathcal{J}^{s}_{K_m,G'}) \le \reg(\mathcal{J}^{s}_{K_m,G})$ for all $s\ge 1$.
\end{Corollary}

Kumar proved in \cite[Theorems 3.5 and 4.6]{MR4405525} that if $G'$ is an induced subgraph of a graph $G$, then  $\reg(\calR(J_{G'}))\le \reg(\calR(J_{G}))$ and $\reg(\calF(J_{G'}))\le \reg(\calF(J_{G}))$ for the regularities of the blowup algebras of classical binomial edge ideals. We can generalize this to the binomial edge ideals of pairs. 

\begin{Theorem}
    \label{induced}
    Under the assumptions in \Cref{set:pair}, we have 
    \begin{equation*}
        \reg(\calR(\calJ_{H_1,H_2}))\le \reg(\calR(\calJ_{G_1,G_2}))
    \end{equation*}
    and
    \begin{equation*}
        \reg(\calF(\calJ_{H_1,H_2}))\le \reg(\calF(\calJ_{G_1,G_2})).
    \end{equation*}
\end{Theorem}
\begin{proof}
    Let $\pi:\KK[\bdX]\to \KK[\bdY]$ be the map defined in the proof of \Cref{power}. 
    We have $\pi(\calJ_{G_1,G_2}^s)=\calJ_{H_1,H_2}^s=\calJ_{G_1,G_2}^s\cap \KK[\bdY]$ for all $s\ge 0$. This fact induces the graded embedding map 
    \[
        \iota: \calR(\calJ_{H_1,H_2}){\into} \calR(\calJ_{G_1,G_2})
    \]
    as well as a graded epimorphism 
    \[
        \pi^*: \calR(\calJ_{G_1,G_2})\onto \calR(\calJ_{H_1,H_2}).
    \]
    Notice that ${\pi}^* \circ \iota$ is the identity map on $\calR(\calJ_{H_1,H_2})$. It follows that $\calR(\calJ_{H_1,H_2})$ is an algebra retract of $\calR(\calJ_{G_1,G_2})$. 
    
    Meanwhile, notice that 
    \begin{align*}
        \calF(\calJ_{G_1,G_2})&\cong \calR(\calJ_{G_1,G_2})\otimes_{\KK[\bdX]} \KK \cong \KK[p_{(e,f)}:e\in E(G_1),f\in E(G_2)]\subseteq \KK[\bdX],\\
    \intertext{and}
        \calF(\calJ_{H_1,H_2})&\cong \calR(\calJ_{H_1,H_2})\otimes_{\KK[\bdX]} \KK 
        \cong \calR(\calJ_{H_1,H_2})\otimes_{\KK[\bdY]} \KK \\
        &\cong \KK[p_{(e,f)}:e\in E(H_1),f\in E(H_2)]\subseteq \KK[\bdY] \subseteq \KK[\bdX].
    \end{align*}
    Therefore, we have a graded embedding map $\iota^\triangle: \calF(\calJ_{H_1,H_2})\into \calF(\calJ_{G_1,G_2})$. On the other hand, by tensoring $\pi^*$ with $\KK$, we have an induced graded epimorphism $\pi^\triangle:  \calF(\calJ_{G_1,G_2}) \onto \calF(\calJ_{H_1,H_2})$.
    Notice that ${\pi}^\triangle \circ \iota^\triangle$ is the identity map on $\calF(\calJ_{H_1,H_2})$. It follows that $\calF(\calJ_{H_1,H_2})$ is an algebra retract of $\calF(\calJ_{G_1,G_2})$. 
    
    To complete the proof, it remains to apply \cite[Corollary 2.5]{MR1789447}.
\end{proof}

\begin{Corollary}
    Let $G$ be a graph and $G'$ be its induced subgraph. Then,  
    \begin{equation*}
        \reg(\calR(\calJ_{K_m,G'}))\le \reg(\calR(\calJ_{K_m,G}))
    \end{equation*}
    and
    \begin{equation*}
        \reg(\calF(\calJ_{K_m,G'}))\le \reg(\calF(\calJ_{K_m,G})).
    \end{equation*}
\end{Corollary}

\begin{Corollary}
    Let $G$ be a graph which contains an induced path with $n$ vertices. Then, we have
    \[
        \reg\left(\frac{S}{\calJ^t_{K_m,G}}\right)\ge 2(t-1)+(n-1)
    \]
    for each $t\ge 1$. Furthermore,
    \begin{equation*}
        \reg(\calR(\calJ_{K_m,G}))\ge (m-1)\floor{\frac{n}{2}}+\floor{\frac{n-1}{2}},
    \end{equation*}
    and
    \begin{equation*}
        \reg(\calR(\calJ_{K_m,G}))\ge 
        \begin{cases}
            (mn-3)-(m\cdot \frac{n}{2}), &\text{if $n$ is even}, \\
            (mn-3)-(m\cdot \frac{n+1}{2}-2), & \text{if $n$ is odd}.
        \end{cases}
    \end{equation*}
\end{Corollary}
\begin{proof}
    These results follow from \Cref{power,thm:power_reg_path,induced,regRees,regfiber}.
\end{proof}

\begin{acknowledgment*}
    The authors are grateful to the software systems \texttt{Macaulay2} \cite{M2} and \texttt{Normaliz} \cite{Normaliz}, for serving as excellent sources of inspiration. This work is supported by the Natural Science Foundation of Jiangsu Province (No.~BK20221353). In addition,  the first author is partially supported by the Anhui Initiative in Quantum Information Technologies (No.~AHY150200) and the ``Innovation Program for Quantum Science and Technology'' (2021ZD0302902). And the second author is supported by the Foundation of the Priority Academic Program Development of Jiangsu Higher Education Institutions.  
\end{acknowledgment*}

\bibliography{BEI} 
\end{document}